\documentclass[12pt]{report}

\usepackage[latin2]{inputenc}
\usepackage{amsmath}
\usepackage[pdftex]{graphicx}
\usepackage{amsfonts}
\usepackage{amssymb}
\usepackage[left=1.5in]{geometry}
\usepackage{geometry}
\usepackage{wrapfig}
\usepackage{setspace}
\usepackage{epsfig}
\usepackage{color}
\usepackage{tikz}
\usepackage{thesis}
\usepackage{float}
\usepackage{url}
\usepackage{hyperref}

\usepackage{amsmath}
\usepackage{url}
\usepackage[pdftex]{graphicx}
\usepackage{multirow}
\usepackage{chapterbib}
\usepackage{amssymb}
\usepackage[titletoc,toc,title]{appendix}
\usepackage{titlesec}
\usepackage{listings,color}
\usepackage{tabularx}
\usepackage[makeroom]{cancel}
\usepackage{titletoc}

\usepackage{graphicx,epsfig,verbatim,enumerate}
\usepackage{amssymb,amsmath,mathrsfs}
\usepackage{ifthen}
\usepackage{caption}
\usepackage{subcaption}
\usepackage{mathtools, algorithm, algpseudocode}
\usepackage{dsfont}
\usepackage[all, cmtip]{xy}

\usepackage{amsthm}
\newtheorem{theorem}{Theorem}[section]
\newtheorem{lemma}{Lemma}[section]

\newcommand{\likelihood}{\mathcal{L}}
\newcommand{\Set}[1]{\left\{ #1 \right\}}

\usepackage{mathtools}

\newboolean{twocolswitch}

\newcommand{\sindex}[1]{}
\newcommand{\nindex}[1]{}

\newcommand{\www}[1]{\url{#1}}

\usepackage{lettrine}

\newcommand{\dee}[1]{\mbox{d}#1}

\usepackage{color}

\newcommand{\R}{\mathbb{R}}
\newcommand{\N}{\mathbb{N}}

\newcommand{\domain}{\Omega}

\newcommand{\eps}{\varepsilon}

\newcommand{\Z}{\mathbb{Z}}
\renewcommand{\P}{\mathbb{P}}

\renewcommand{\domain}{\Omega}
\newcommand{\intdomain}{\int_{\domain}}
\DeclareMathOperator*{\argmin}{arg\,min}

\newcommand{\functionspace}{\mathcal{V}}
\newcommand{\loss}{\mathcal{L}}

\graphicspath{{./thesisfigures/}}

\title{Some results on a class of functional optimization problems}
\author{David Rushing Dewhurst}
\defensedate{March 23rd, 2018}
\thesis                     
\masterscience              
\apma                          
\maygrad                           
\advisor{Chris Danforth, Ph.D.}
\readerone{Peter Dodds, Ph.D.}
\readertwo{Brian Tivnan, Ph.D.}
\chair{Bill Gibson, Ph.D.}
\dean{Cynthia J. Forehand, Ph.D.}

\begin{document}

\maketitle
\pagenumbering{roman}

\begin{abstract}
	
\vspace{10mm}
We first describe a general class of optimization problems that describe many natural, economic, and statistical
phenomena.
After noting the existence of a conserved quantity in a transformed coordinate system, we outline several 
instances of these problems in statistical physics, facility allocation, and machine learning. 
A dynamic description and statement of a partial inverse problem follow.
When attempting to optimize the state of a system governed by the generalized equipartitioning principle, it is 
vital to understand the nature of the governing probability distribution.
We show that optimiziation for the incorrect probability distribution can have catastrophic results, e.g., infinite 
expected cost, and describe a method for continuous Bayesian update of the posterior predictive distribution 
when it is stationary.
We also introduce and prove convergence properties of a time-dependent nonparametric kernel density estimate (KDE)
for use in predicting distributions over paths.
Finally, we extend the theory to the case of networks, in which an event probability density is defined over nodes and 
edges and a system resource is to be partitioning among the nodes and edges as well.
We close by giving an example of the theory's application by considering a model of risk propagation on a power grid.

\end{abstract}

\begin{dedication}
	
	\flushright{ {\it in memory of} 
\\
\vskip 2em
David Conrad Dewhurst (1918-2005)
\\
\vskip 2em
Eloise Linscott Dewhurst (1922-1999)
\\
\vskip 2em
Margaret Jones Hewins (1923-2004)
\\
\vskip 2em
	\textit{A formless chunk of stone, gigantic, eroded by time and water, though a hand, 
	a wrist, part of a forearm could still be made out with total clarity.\\
	-R. Bola\~{n}o}}
	
\end{dedication}

\begin{acknowledgements}

  Where to begin? 
First, to my advisors: Chris Danforth, Peter Dodds, Brian Tivnan, and Bill Gibson. 
They have helped me in nondenumerable ways over the years I have known them,
With this thesis, of course,
	but also with various issues---``I need to get paid!", ``My students hate me!", ``The data isn't there!", and
	other fun incidents---as well in terms of friendship; 
	our mutual relationships are marked by the essential requirement that I refer to them exclusively by
	their last names.
	Danforth was an indispensible help in all things administrative, as well as being an incredible professor
	in the three courses I took with him. His skills in the power clean are as remarkable as his mastery 
	of dynamical systems is deep.
	I wouldn't be sane without Dodds's help and friendship, on which I have come to rely.
	``We really do have to go home," he and I often jointly remark while sitting in his office, and continue to sit for 
	several hours more. 
	Tivnan, who in addition to being a thesis committee member is also my supervisor at the Mitre Corporation, 
	has helped guide me down the path of righteousness for the past year and a half 
	without fail. His knowledge of esoteric movie quotes is also impressive. 
	I have known Gibson for the longest of the four, and it was he who provided me with the highest-quality
	undergraduate economics experience for which one could ask. His ability to provide both calming advice
	and excoriating insult, almost simultaneously, is unrivaled; I would not be the man I am without his
	guidance.
	To all four of you gentlemen: thank you, truly.
	\medskip

	\noindent
	To the entire graduate faculty and administration with whom I've interacted: thank you for your patience
	as I, a fundamentally nervous person, bombarded you with questions. I am particularly thankful to Sean
	Milnamow for putting up with my ceaseless queries regarding financial aid and to Cynthia Forehand
	for having the fortitude to admit me to graduate study in the first place. 
	To James Wilson, Jonathan Sands, and Richard Foote: thank you for your tireless effort in teaching
	me real and complex analysis. The memories of staying up late at 
	night to finish my assignments will stay with me for the rest of my life.
	It is rare to realize that you will miss something forever as it is passing,
	but you have given me those moments
	and I will be forever grateful for that in a way I cannot express.
	To Marc Law, whose undergraduate economics courses have shaped the way I view the world: your words and 
	lessons will be felt in everything I do in public life.
	\medskip

	\noindent
	To my fellow graduate students, Ryan Grindle, Ryan Gallagher, Kewang, Damin, Shenyi, Francis,
	Marcus, Sophie, Michael, Rob, and Ben: 
	thank you for making my coursework enjoyable and sharing ideas, recipes,
	and laughter with me. 
	To my calculus classes I've taught: I cannot thank you enough. You have made me work and I enjoyed 
	every second of it. Some of the happiest moments of my life came when you told me that my teaching
	made you love mathematics again, or for the first time. 
	\medskip

	\noindent
	To my good friends, Colin van Oort and John Ring:
	Let the saga continue.
	To Alex Silva: I'll be home soon.
	To my parents, Sarah Hewins and Stephen Dewhurst: thank you for teaching me how to write and how to
	think. To my fiance, Casey Comeau: you know what I'm going to say.
	And to K.: just hang on...

\end{acknowledgements}

\newpage

\tableofcontents
\newpage

\listoffigures
\newpage

\listoftables
\newpage

\doublespacing
\pagenumbering{arabic}

\chapter{The generalized equipartitioning principle}

\begin{quote}
We describe a general class of optimization problems that describe many natural, economic, and statistical
phenomena.
After noting the existence of a conserved quantity in a transformed coordinate system, we outline several 
instances of these problems in statistical physics, facility allocation, and machine learning. 
A dynamic desription and statement of a partial inverse problem follow, along with questions for further
research.
\end{quote}

\section{Introduction and background}
Methods for solving continuous optimization problems are almost as old as calculus, which was developed in the 17th century.
Johann Bernoulli posed and solved the famous problem of determining the curve of minimal travel time traced out by a particle 
acting only under the influence of gravity, otherwise known as the {\it brachistocrone} problem. 
A few years before, Isacc Newton (who also solved the brachistocrone problem) posed the problem of determining a solid
of revolution that experience minimal resistance when rotated through fluid.
The number of problems of this nature under consideration by the mathematical community were greatly increased 
with the advent of analytical mechanics, developed by d'Alembert, Lagrange, and others. 
They realized that Newton's classical mechanics, in which the motion of objects is described via three fundamental equations
related momentum, acceleration, and total force, could be re-expressed using the potential and kinetic energy of particles. 
This discovery revolutionized physics and made way for the formal development of the calculus of variations, which we 
use extensively in this paper.
William Rowan Hamilton further generalized this principle in his further reformulation of classical mechanics, leading
(eventually) to the formulation of quantum mechanics.
\medskip

\noindent
Optimization under uncertainty has a similarly illustrious history. 
The first academic mention of this concept appears to be due to Blaise Pascal in his formulation of the 
philosophical concept that, in choosing whether or not to believe in God, humans are performing an expected utility 
maximization procedure (though he did not state it in this manner explicity). 
Daniel Bernoulli also addressed the maximization of expected utility explicitly, providing one of the first examples of the 
modern understanding of utility functions. 
Interest in this subject flowered in the 20th century, with von Neumann and Morgenstern publishing a set of ``axioms" 
concerning rational decision-making under uncertainty that is still a foundation of economic theory today.
\medskip

\noindent
The combination of continuum problem formulation and optimization under uncertainty is a relatively new development, 
as to be well-formulated it required the development of measure-theoretic probability which was not truly complete 
until Kolmogorov's work in 1933.
The concept of finding an optimal decision field $S(x)$, where $x \in \Omega \subseteq \R^n$ and the optimizer 
attempts to mitigate events occuring according to the probability measure $P(x)$, is largely confined to statistics
(in the field of empirical risk minimization, c.f.\ Sec.\ \ref{sec:machine-learning}) and economics (in the field
of microeconomics, and particularly in the field of decision theory).
Practically, of course, it is understood heuristically by practitioners in professional fields that are fundamentally 
concerned with either profiting by purchasing and selling risk or with mitigating risk exposure, such as finance, 
insurance, and medicine.
Even where problem domains are not continuous (c.f. Chapter 3) a continuum formulation can often ease analysis; the
methods of functional analysis that underlie the contiuum formulation of problems are often applicable to problems formulated 
on lattices and other discrete structures. 
The core utility of the method lies in its ability to generate, via the machinery of the variational principle, 
sets of algebraic or differential equations that can be solved using well-known analytical tools and numerical routines.
\medskip

\noindent
Our work collates, extends, and unifies work done in three disparate areas: statistical physics, microeconomics and 
operations research, and machine learning. 
Much as neural networks can be studied (as a canonical ensemble) from the point of condensed matter theory, 
we have found that a particular class of continuum optimization problems (described in Sec. \ref{sec:theory})
can be described neatly via a simple generalized equipartitioning principle; the above problem domains are
contained wholly within this general class of problems.
\begin{figure}
\centering
	\includegraphics[width=\textwidth]{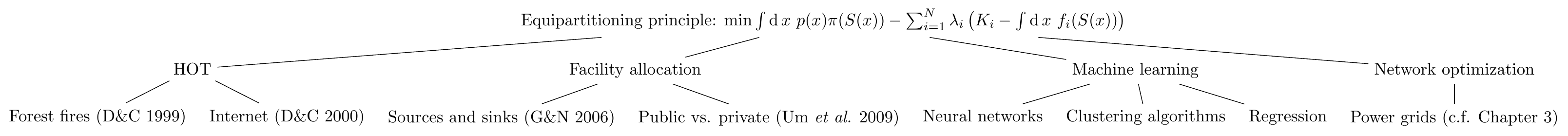}
	\caption{A partial scope of the hierarchy of problems subsumed by the generalized equipartitioning
	principle.
	Of course, not all possible realizations of this general problem are treated here. 
	In fact, this is what makes this formulation so powerful: any problem that can be recast
	in this formulation will have an invariant quantity (Eq. \ref{eq:invariant}), leading to deep insights 
	about the nature of the problem and its effect on the system in which it is embedded.}
	\label{fig:problem-scope}
\end{figure}
Figure \ref{fig:problem-scope} gives a partial scope of the hierarchy of problems treated by the generalized
equipartitioning principle. 
Under this unifying theory, we posit the existence of isomorphisms between the problems of 
minimizing the risk of a forest fire or cascading failure in the Internet, understanding the distribution 
of firms in a geographic location, and finding functions that best fit a particular dataset---tasks that {\it a 
priori} seem almost entirely unrelated.
\medskip

\noindent
We outline the theory of the generalized equipartitioning principle below and describe some classes of problems
to which it applies. 
In particular, we note that the generic supervised machine learning problem is a subclass of this formalism;
algorithms constructed for use in these problems could reasonably be applied to solve physical problems (such 
as highly-optimized tolerance and facility allocation) and, conversely, physical techniques developed in 
these areas can be tailored to solve classification and regression problems in machine learning.
Section \ref{sec:theory} gives the general theoretical results, section \ref{sec:applications} 
gives applied context, section \ref{sec:dynamic} describes the optimal allocation of resources under 
the influence of time-dependenbt coordinates, and section \ref{sec:inverse} describes the pseudo-inverse 
problem of finding the distribution for which a system was most likely optimized and suggests a method for its solution. 

\section{Theory}\label{sec:theory}

Let $\domain \subseteq \R^N$ and let $p:\domain \rightarrow \R$ be a probability density function,
$S:\domain \rightarrow \R$ be a resource allocation function in 
$L^1(\domain) \cap L^2(\domain)$, and 
$\pi: \R \rightarrow \R$ be a differentiable net benefit function.
Consider the optimization problem
\begin{equation}\label{eq:general-optim}
\begin{aligned}
	\max &\intdomain \dee{x}\ p( x) \pi(S( x)) \\
	&\quad \text{s.t.\ } \intdomain \dee{x}\ f_i(S( x)) = K_i,\ i=1,...,M.
\end{aligned}
\end{equation}
The action associated with this problem is 
\begin{equation}\label{eq:general-action}
\begin{aligned}
	J = \intdomain \dee{x}\ p( x) &\pi(S(x)) \\ 
	&-\sum_{i=1}^M\lambda_i\left(K_i - \intdomain \dee{x}\ f_i(S(x))\right),
\end{aligned}
\end{equation}
where $f_i: \R \rightarrow \R$ are constraint functions.
The optimal state of the system is given by $\delta J / \delta S = 0$,
which here takes the form
\begin{equation}\label{eq:static-eq}
	p(x)\frac{\partial \pi}{\partial S} + 
	\sum_{i=1}^M \lambda_i \frac{\partial f_i}{\partial S} = 0.
\end{equation}

We consider diffusion of the probability density $p(x)$.
Solution of the diffusion equation with $q(x,0) = p(x)$ and the Neumann boundary conditions on $\Omega$ 
(necessary, as we must have zero probability flux) defines a functional transform $p(x) 
\xmapsto{\mathcal{D}} \text{Unif}$, where $\text{Unif}$ is the uniform distribution on $\Omega$. 
The resultant steady state of the diffusion equation is $\langle p \rangle$. 
Transforming $x \mapsto \mathcal{D}(x)$ and substitution into Eq. \ref{eq:static-eq} results in 
the expression 
\begin{equation}\label{eq:invariant}
	\langle p \rangle =
	-\frac{\lambda_{\ell}\frac{\partial f^{(\ell)}}{\partial S}}{\partial \pi / \partial S},	
\end{equation}
(where we are now employing the Einstein summation convention),
showing that the quantity 
$\lambda_{\ell}\frac{\partial f^{(\ell)}}{\partial S}/\frac{\partial \pi}{ \partial S}$ remains 
constant across $\Omega$.
This quantity is seen to be a function of the constraint-weighted marginal benefit;
in the transformed coordinate system, marginal benefit is inversely proportional to (constant) event probability
and proportional to the weighted constraint gradient.

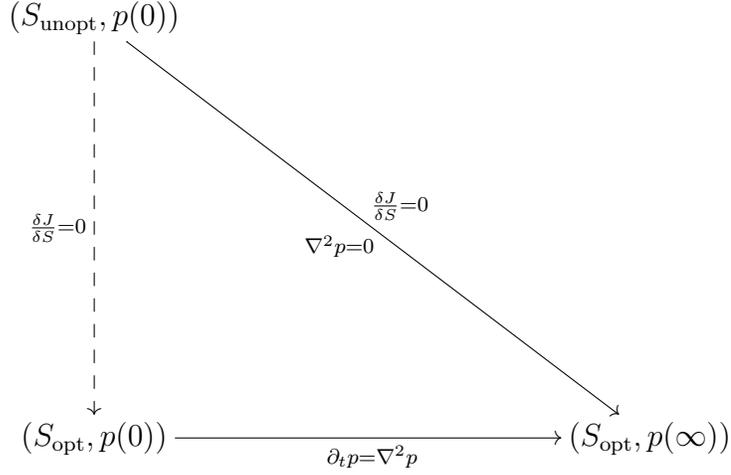
\begin{figure}
\begin{equation*}
\xymatrix@=12em{
	(S_{\text{unopt}}, p(0)) \ar@{-->}[d]_{\frac{\delta J}{\delta S}=0}
\ar[dr]_{\nabla^2 p = 0}^{\frac{\delta J}{\delta S} = 0}
	& ~ \\
	(S_{\text{opt}},p(0)) \ar[r]_{\partial_t p = \nabla^2 p} & (S_{\text{opt}}, p(\infty))
}
\end{equation*}
	\caption{
		A diagrammatic representation of the optimization process. 
		The edge with $\nabla^2 p = 0$ and $\delta J/\delta S = 0$ gives an immediate transform from 
		the initial unoptimized system $(S_{\text{unopt}}, p(0))$ to the optimized system in the coordinates
		$x \mapsto \mathcal{D}(x)$, written $(S_{\text{opt}}, p(\infty))$.
		The link from $(S_{\text{unopt}}, p(0)) $
		to $(S_{\text{opt}}, p(0))$ shows the relaxation to the optimal state 
		given by $\delta J/\delta S =0$ in the natural (un-diffused) coordinate system. 
		Subsequently diffusing the coordinates via solution of $\partial_t p = \nabla^2 p$ 
		again gives the diffused and optimized state $(S_{\text{opt}}, p(\infty))$. 
				}
\label{fig:diagram}
\end{figure}

\section{Application}\label{sec:applications}
We consider three systems in particular (with a note on the equipartition theorem first):
Doyle and Carlson's models of highly optimized tolerance (HOT)
\cite{carlson1999highly,carlson2000highly};
Gastner and Newmans's approach to the optimal facility allocation
($k$-medians) problem \cite{gusein1982bunge,gastner2006optimal}; and a generalized form  
of supervised machine learning \cite{zou2005regularization}.

\subsection{Statistical mechanics: the equipartition theorem}
The well-known equipartition theorem is a simple consequence of this formalism. 
Let $\P$ be a probability measure on phase space and let $\dee \Gamma = \prod_i \dee x_i \dee p_i$ be 
phase space differential. 
Denoting the Hamiltonian of the system by $\mathscr{H}(p,x)$, the integral to minimize is given by
\begin{equation}
	\int \dee \P(\Gamma) \mathscr{H}(p,x) - \lambda 
	\left(1 - \int \dee \Gamma\ e^{-\beta \mathscr{H}(p,x)}  \right).
\end{equation}
Performing the optimization gives the value of the Hamiltonian at the optimum to be
\begin{equation}
	\mathscr{H}(p,x) = -T \log \frac{\P(\Gamma)}{Z},
\end{equation}
where $T$ is the thermodynamic temperature and $Z$ is the partition function. 
The equipartition principle follows via integration by parts of the constraint equation. 
The usual connections to information theory also follow: denoting the information contained in the 
random variable
$\Gamma$ by $\mathcal{I}(\Gamma) = - \log \P(\Gamma)$, 
we can rewrite the minimal energy Hamiltonian as $\mathscr{H}(p,x) = T(\log Z + \mathcal{I}(\Gamma))$.
Substitution into the objective function gives
$$
\begin{aligned}
	\int \dee \P(\Gamma) \mathscr{H}(p,x) &= \int \dee \P(\Gamma) 
	\left[ T(\log Z + \mathcal{I}(\Gamma)) \right] \\
	&= T(\log Z + H(\Gamma)),
\end{aligned}
$$
where $H(\Gamma)$ is the entropy of $\Gamma$.
In physics the probability measure $\P$ is the uniform distribution over state space.
In the systems considered below this is almost universally not the case; indeed, the interesting behavior in 
such systems is partially generated by the inhomogeneity of the probability distributions over their 
``phase space".

\subsection{HOT}
Carlson and Doyle introduced the idea of highly-optimized tolerance (HOT) in a series of papers in 1999 and 2000 
\cite{carlson1999highly, carlson2000highly}.
Of the previous work known to the author that is related to this paper, Carlson and Doyle came closest to uncovering the 
true generality of the generalized equipartitioning principle. 
They found that many physical systems are created, via evolution or design, to minimize expected cost due to events 
ocurring with some distribution $p(x)$ over state space $x \in \Omega \subseteq \R^2$.
A purely physical argument relating event cost to area affected by the event, $C \propto A^{\alpha}$,
and subsequently relating affected event area to the amount of system resource in the area, $A \propto S^{-\beta}$, 
gave the expected cost to be $\int \dee x\ p(x) C(x) \propto \int \dee x\ p(x) S(x)^{-\alpha \beta}$.
Carlson and Doyle supposed the constraint on the system took the form of a maximum available amount of the system
resource, $K = \int \dee x\ S(x)$. 
The integral to minimize is thus
\begin{equation}
	\int_{\Omega}\dee{x}\ p(x)S^{-\gamma}(x)
	- \lambda\left(K - \int_{\Omega}\dee{x}\ S(x)\right),
\end{equation}
giving the optimum $S(x) \propto p(x)^{\frac{1}{\gamma + 1}}$.
They showed that this result is reflected empirically in the distribution of forest fire breaks.
Figure \ref{fig:hot} shows evolution to the HOT state as proposed in \cite{carlson1999highly}.
The evolution results in structurally-similar final states regardless of spatial resolution, as shown in 
the figure.
\begin{figure}[!htp]
\centering
	\includegraphics[width=\textwidth]{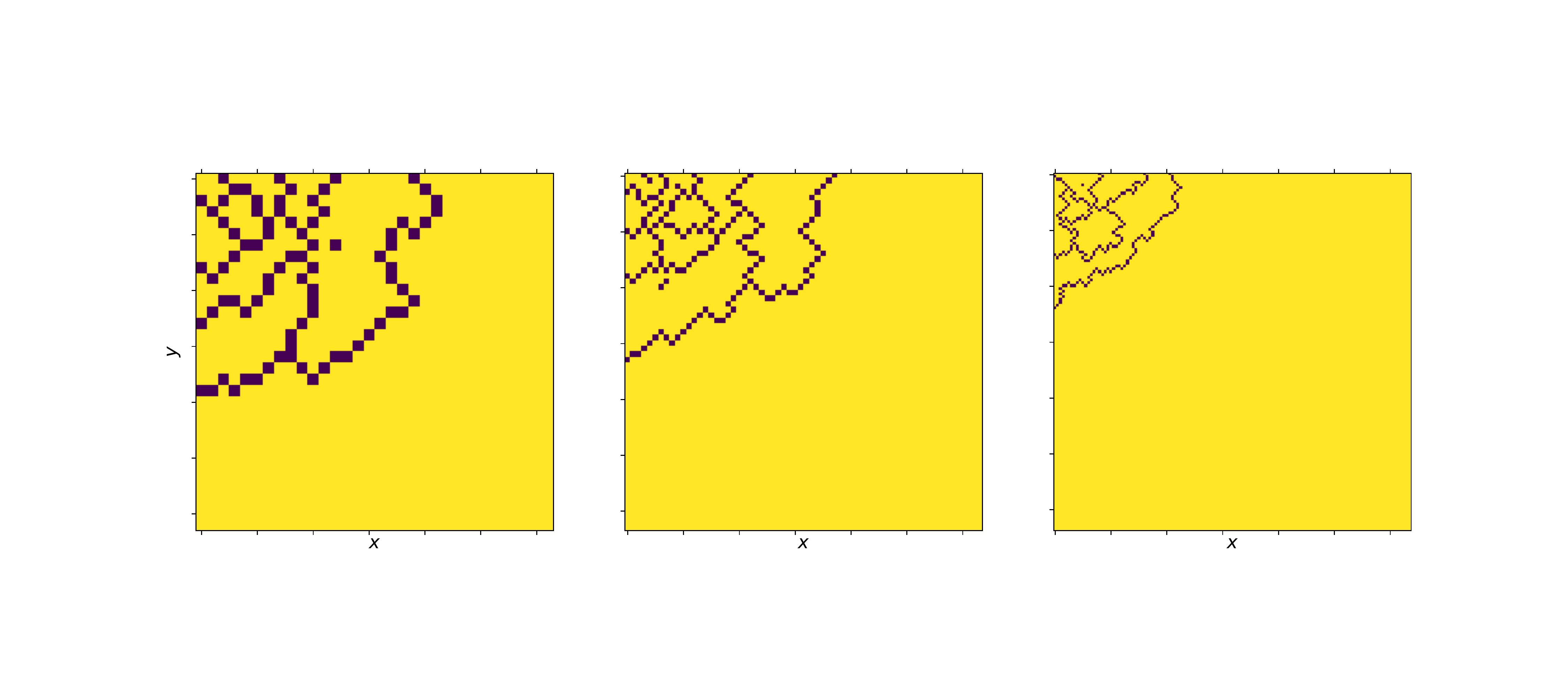}
	\caption{Realizations of evolution to the HOT state as proposed in Carlson and Doyle.
	The ``forest" is displayed as yellow while the ``fire breaks" are the purple boundaries.
	The evolution to the HOT state results in structurally-similar low-energy states
	regardless of spatial resolution, as shown here. 
	From left to right, 32 $\times$ 32, 64 $\times$ 64, and 128 $\times$ 128 grids. 
	The probability distribution is $p(x,y) \propto \exp( -(x^2 + y^2) )$ 
	defined on the quarter plane with the origin $(x,y) = (0,0)$ 
	set to be the upper left corner.}
	\label{fig:hot}
\end{figure}

\subsection{Facility placement}
A classical problem in geography and operations research is to minimize the median (or average) distance
between facilities in the plane.
This problem, known as the $k$-medians (or $k$-means) problem, is $\mathbf{NP}$-hard, so approximation algorithms and
heuristics are often used to approximate general solutions. 
Considering the objective function corresponding to the median distance between facilities,
$\int p(x) \min_{i = 1,...,k} ||x - x_i ||\ \dee x$, Gastner and Newman 
found the optimal solution in two dimensions to scale as $S(x) \propto p(x)^{2/3}$, where here $S$ is 
interpreted as facility density ($S \propto A^{-1}$) and $p$ as population density. 
The $N$-dimensional version of this problem follows by minimizing the integral
\begin{equation}
	\int_{\Omega}\dee{x}\ p(x) V(x)^{1/N}  - 
	\lambda\left( K - \int_{\Omega}\dee{x}\ V(x)^{-1} \right),
\end{equation}
which leads to a solution of the form $S(x) \propto p(x)^{\frac{N}{N+1}}$, notably resulting in
$\gamma = 2/3$ scaling in $N = 2$ dimensions (as found by Gastner and Newman) and 
$\gamma = 3/4$ in $N=3$ dimensions.
Considering instead the {\it average} (least squares) distance between facilities results in the minimization of 
\begin{equation}
	\int_{\Omega}\dee{x}\ p(x) V(x)^{2/N}
	- \lambda\left( K - \int_{\Omega}\dee{x}\ V(x)^{-1} \right),
\end{equation}
resulting in optima given by $S(x) \propto p(x)^{\frac{N}{N+2}}$, e.g., $\gamma = 1/2$ in $N = 2$ dimensions
and $\gamma = 3/5$ in $N = 3$ dimensions.

\subsection{Machine learning}\label{sec:machine-learning}
We give a short overview of the general supervised machine learning problem in $\R^N$.
We observe data $x \in \R^N$ and wish to predict values $y \in \R$, some of which we also observe,
based on these data.
In general, we fit a model $S(x)$ to the data and evaluate its error against $y$ via a loss function
$\loss(y, S(x))$. 
The data is distributed $x \sim p(x)$, although in any applied context this distribution is never known. 
Thus the general \textit{unconstrained} problem is to search a particular space of functions $\functionspace$ 
for a function $S^*$ such that
\begin{equation}
	S^*(x) = \argmin_{S \in \functionspace} \int_{\R^N}\dee{x}\ p(x)\loss(y,S(x)).
\end{equation}
The empirical approximation of this problem often goes by the moniker of empirical risk minimization.
\medskip

\noindent
It is often desireable to impose restrictions on the function $S^*$. 
For example, one may wish to limit the size of the function as measured by its $L^1$ or $L^2$ norms, or 
to mandate that the function assign a certain value to a particular subdomain $D \subseteq \R^N$.
A common example is that of the \textit{elastic net}, introduced by Zou and Hastie in 
\cite{zou2005regularization}, which penalizes higher $L^1$ and $L^2$ norms. 
The solution to the corresponding constrained problem is thus
\begin{equation}\label{eq:regularized}
	\begin{aligned}
		S^*(x) =&\argmin_{S \in L^1(\R^N) \cap L^2(\R^N)} 
		\int_{\R^N}\dee{x}\ p(x)\loss(y,S(x))  \\
		&\qquad\quad- \lambda_1\left(M_1 - \int_{\R^N}\dee{x}\ |S(x)|\right)\\
		&\qquad\quad-\lambda_2\left(M_2 - \int_{\R^N}\dee{x}\ S(x)^2 \right).
	\end{aligned}
\end{equation}
\medskip

\noindent
This formulation, while perhaps unduly formal, does encapsulate the entirety of this field, from the simplest
of examples (linear regression) to the most complicated (deep neural networks). 
Restricting the function space $\functionspace$ to be linear functions $X \mapsto X\beta$, approximating 
$p(x) \approx \frac{1}{N}\sum_{i=1}^N \mathds{1}_{x_i}(x)$, and solving the unconstrained version of the problem 
using the mean squared error loss function 
gives $\min_{\beta} \frac{1}{N}\sum_{i=1}^N (Y_i - x_i^T \beta)^2 
= \min_{\beta}||Y - \beta X||_2^2$, which is easily 
seen to be the canonical ordinary least squares problem, while incorporating $L^2$ regularization as above 
gives $\min_{\beta} ||Y - \beta X||_2^2 + \lambda ||\beta||_2^2$, the ridge regression problem 
\cite{zou2005regularization}.
On the other end of the model complexity spectrum, approximating $p(x)$ via a variational autoencoder 
\cite{kingma2013auto} and subsequently fitting a regularized deep neural network 
perhaps most closely approximates the 
true, continuum form (Eq.\ \ref{eq:regularized}) due to the function-approximating properties of neural networks.
The ability to closely approximate the true form of the action integral may explain these models' 
success in many forms of classification and regression \cite{cybenko1989approximation, hagan1994training}.
We note also that the isometry between physical problems, such as HOT, and supervised machine learning 
problems mean that 
algorithms developed for the latter may be used to great utility in the former; instead of laboriously 
constructing highly-optimized forest fire breaks via artificial evolution, as done in \cite{carlson1999highly}, 
or using computationally-intensive simulated annealing algorithms to allocate facilities, as in 
\cite{gastner2006optimal}, one may simply use a fast approximation algorithm, 
such as $k$-medians or SVM, to obtain the same result.
Conversely, insights from physical problems could be used to create new machine learning algorithms or 
paradigms, e.g., 
in the inference of more effective loss functions for regression or classification problems.

\subsection{Empirical evidence}
\label{sec:evidence}
We provide empirical evidence for the hypothesis by constructing realizations of the diffusion transform acting
on disparate datasets and for a variety of probability distributions.
Figure \ref{fig:equipartition} displays the equipartitioning process as applied to the facility allocation 
problem (using simulated data) and a binary classification problem implemented via 
support vector machine (SVM) (using the Wisconsin 
breast cancer dataset \cite{wolberg1993breast}).
\medskip
\begin{figure*}[!htp]
  \centering	
	\includegraphics[width=0.75\textwidth]{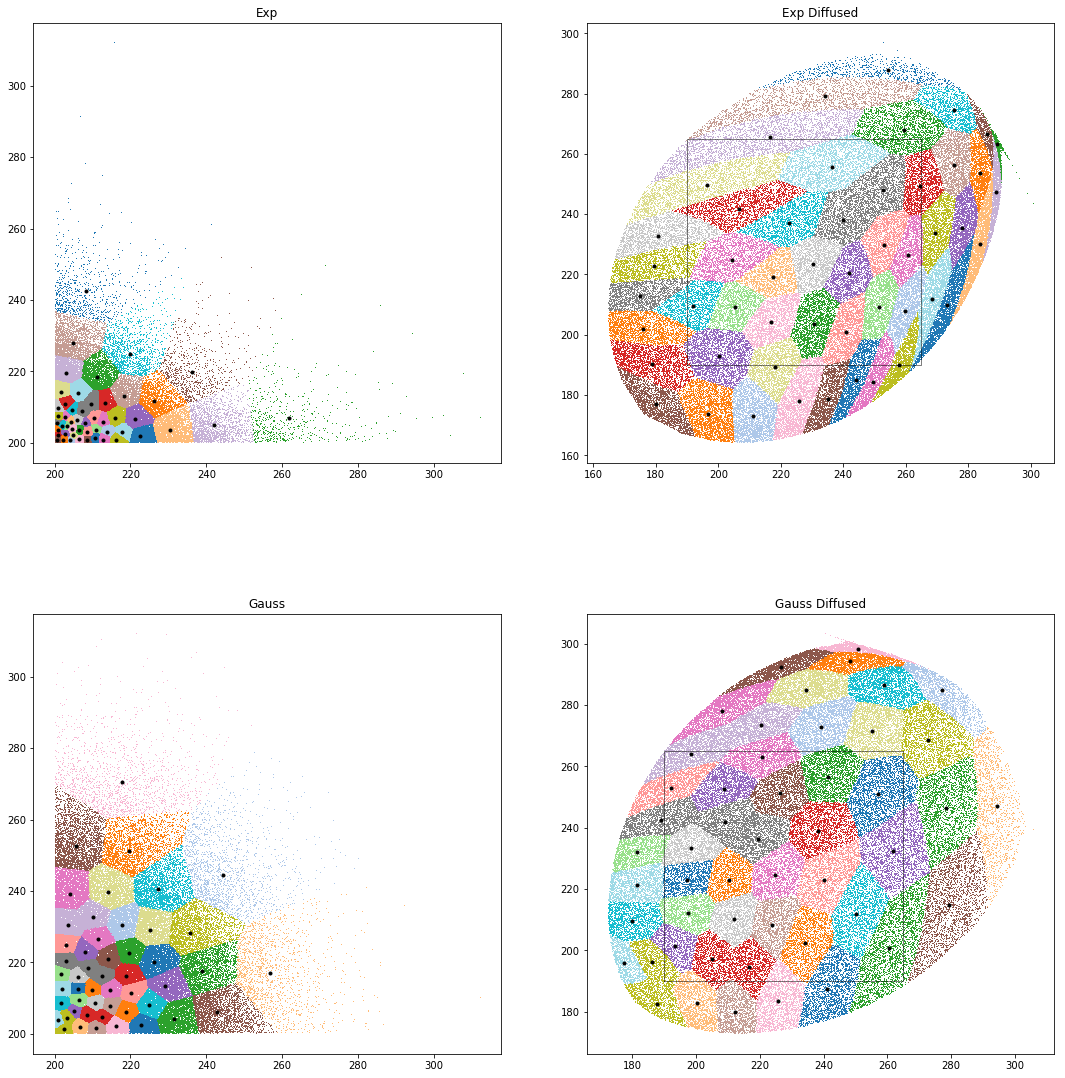}  \\
	\includegraphics[width=0.75\textwidth]{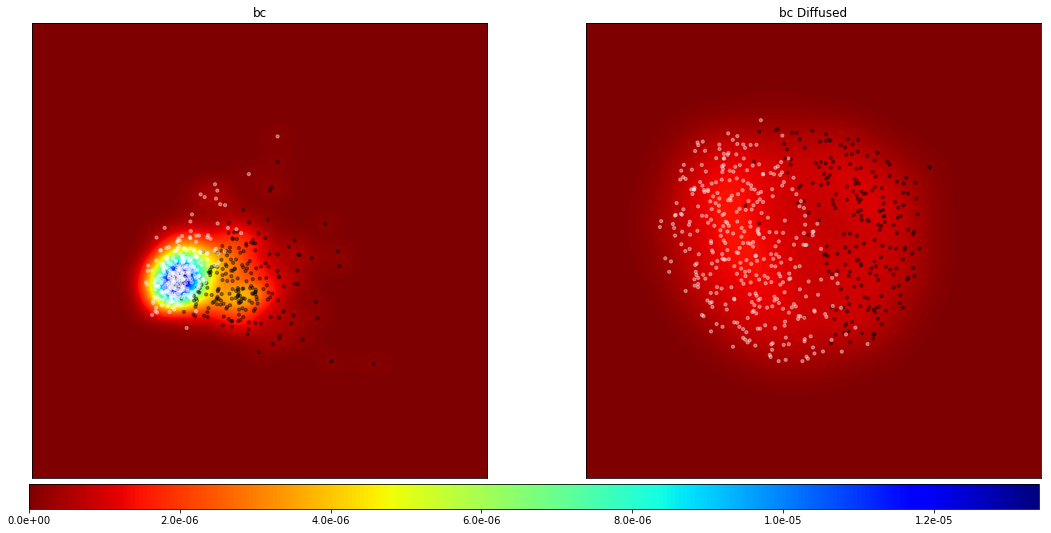}  
  \caption{
    \textbf{ The equipartitioning principle }
The equipartitioning principle as observed in facility allocation and machine learning. 
	Here, the support vector machine (SVM) algorithm is used for binary classification and class labels 
	are displayed.
The SVM loss function, known as the {\it hinge loss}, is given in 
	its continuum form by $\mathcal{L}(S) = \max \{0, 1 - Y(X)S(X)\}$, 
which is commonly minimized subject to $L^1$ and $L^2$ constraints as discussed above.
  }
  \label{fig:equipartition}
\end{figure*}

\noindent
The top and middle figure display the result of heuristically solving the $k$-medians problem using
the standard expectation-maximization (EM) algorithm.
Beginning with two different distributions (Gaussian and exponential) defined on the quarter plane 
$\Set{(x,y) \in \R^2: x \geq 0,\ y\geq 0}$, 
the EM algorithm is run with the specification that $N = 50$ locations be placed---that is, the 
constraint is $\int \dee x\ A(x)^{-1} = 50$ in the manner of \cite{gastner2006optimal}, 
as area is two-dimensional volume---and 
optimized locations are shown on the left.
The diffusion equation is then solved numerically (using Fourier cosine series) and the facilities' locations are 
transformed via the resulting diffusion transformation; the results are plotted on the right.
\medskip

\noindent
The bottom figure displays the result of learning a binary classification of subjects into breast cancer $(Y = 1)$
/ no breast cancer ($Y=0$) categories using SVM.
As the probability of a subject having breast cancer is (in untransformed space) much less probable than 
not having breast cancer, the left plot, which shows the result of the classification, has a high density of 
subjects $Y = 0$ and a much more diffuse density of $Y = 1$.
When the diffusion transform is applied and the result plotted on the right, the density is nearly equalized. 
We note that, in diffused coordinates, the decision boundary of the SVM is given by a vector that splits the 
data essentially in half, consistent with the imposed constraint that there be only two classes; 
the classes are equipartitioned across the space.

\section{Dynamic allocation}\label{sec:dynamic}
Thus far we have restricted our attention to a static problem; implicitly we have assumed that there is
no cost associated in transporting $S$ from location to location. 
While transport costs can safely be neglected in many scenarios, still others remain in which transport
is a primary consideration. 
We now generalize the above result to a dynamic result for the time-dependent field $S(x,t)$ 
where $x$ is some finite-dimensional vector that may depend on time (the case of moving coordinates 
is treated explicitly). 
We will consider only the cost minimization problem as the above-treated net-benefit maximization is
essentially identical.
Assume a cost function of the form 
\begin{equation}
	C_{\text{total}} = C_{\text{transport}} + \langle C_{\text{event}}\rangle 
	+ C_{\text{constraint}},
\end{equation}
where the expectation is taken over all $(x,t)$ with respect to $p(x,t)$. 
We will assume that the transport costs are proportional to a suitably generalized notion of the work
done on the resource in the process of moving it; letting $W$ be the work, we suppose
\begin{equation}
	C_{\text{transport}} \propto W^{\alpha/2} \propto \frac{1}{\alpha}\left(\frac{DS}{Dt}\right)
	^{\alpha},
\end{equation}
where $1 \leq \alpha \leq 2$ and $\frac{D}{Dt} = \frac{\partial}{\partial t} 
+ \frac{dx_i}{dt}\frac{\partial}{\partial x_i}$, the material derivative, which is 
the correct generalization of the derivative in the case of moving coordinates $\frac{dx}{dt} = g(x,t)$.
(The reader should note that when coordinates are stationary this becomes the standard time operator 
$\frac{\partial}{\partial t}$ as usual.)
We seek a minimum of the action $\int\dee{t}\ \int\dee{x}\ \mathscr{L}$,
where $\mathscr{L}$ is the Lagrangian density given by
\begin{equation}\label{eq:lagrangian-dynamic}
	\mathscr{L} = \frac{1}{\alpha}\left(\frac{DS}{Dt}\right)^{\alpha}
	 - p(x,t)L(S(x)) - \lambda_{\ell}f^{(\ell)}(S(x)), 
\end{equation}
Introducing the generalized momentum $\Pi = \frac{\delta \mathscr{L}}{\delta D_tS}$
the Hamiltonian density is given by 
\begin{equation}
	\begin{aligned}
		\mathscr{H} &= \frac{DS}{Dt}\Pi - \mathscr{L}\\
		&= \frac{\alpha - 1}{\alpha} \Pi^{\frac{\alpha}{\alpha - 1}} + p(x)L(S(x)) 
		+\lambda_{\ell}f^{(\ell)}(S(x)).
	\end{aligned}
\end{equation}
Hamilton's field equations are $\frac{D\Pi}{dt} = - \frac{\delta \mathscr{H}}{\delta S}$ and
$\frac{DS}{Dt} = \frac{\delta \mathscr{H}}{\delta \Pi}$, resulting in 
\begin{align}\label{eq:hamiltons}
	\frac{D\Pi}{Dt} &= - p(x,t)\frac{\partial L}{\partial S} - \lambda_{\ell}
	\frac{\partial f^{(\ell)}}{\partial S} \\
	\frac{DS}{Dt} &= \Pi^{\frac{1}{\alpha - 1}}.
\end{align}
A proof of correctness is given in Appendix A.
Two cases bear special mention. 
When $\alpha = 2$ and coordinates are stationary, these are the standard Hamiltonian field equations 
$\frac{\partial \Pi}{\partial t} = - \frac{\delta \mathscr{H}}{\delta S}$ and 
$\frac{\partial S}{\partial t} = \Pi$, resulting in the expected oscillatory behavior of $S$ in time.
When $\alpha = 1$ and coordinates are stationary, we have 
\begin{equation}
	\frac{D\Pi}{Dt} = \frac{\partial \Pi}{\partial t} 
	= (\alpha - 1)\left(\frac{\partial S}{\partial t}\right)^{\alpha - 2}
	\frac{\partial^2S}{\partial t^2} = 0,
\end{equation}
along with $\lim_{\alpha \rightarrow 1^+} \frac{DS}{Dt} 
= \lim_{\alpha \rightarrow 1^+}\Pi^{\frac{1}{\alpha - 1}} \rightarrow +\infty$,
translating to infinitely fast allocation of $S$ with the equilibrium state given by
$p(x)\frac{\partial L}{\partial S} + \lambda_{\ell}\frac{\partial f^{(\ell)}}{\partial S} = 0$---
in other words, the static optimum.
Thus the static theory is entirely recovered as a special case of the current structure, 
as expected given that $\mathscr{L} \mapsto \mathscr{L} + \text{{\bf div}}\ S$ gives
rise to the same Euler-Lagrange equations as $\mathscr{L}$.
\medskip

\noindent
We note also that disspative forces can be introduced via the Rayleigh function 
$V = \int \dee{x}\ \frac{k(x)}{2}\left( \frac{DS}{Dt} \right)^2$, whereupon 
the nonconservative force $F = - \frac{\delta V}{\delta \dot{S}} = -k(x)\frac{DS}{Dt}$
is incorporated into the Euler-Lagrange equation. 
In the case where $dx_i/dt = 0$ and $\alpha = 2$, this becomes
\begin{equation}\label{eq:damped-euler-lagrange}
	\frac{\partial^2 S}{\partial t^2} + k(x)\frac{\partial S}{\partial t} = 
	- p(x)\frac{\partial L}{\partial S} - \lambda_{\ell}\frac{\partial f^{(\ell)}}{\partial S}.
\end{equation}
In some cases the overdamped limit of Eq.\ \ref{eq:damped-euler-lagrange},
\begin{equation}\label{eq:overdamped-euler-lagrange}
\frac{\partial S}{\partial t} =- \frac{1}{k(x)}\left( p(x) \frac{\partial L}{\partial S}
+ \lambda_{\ell}\frac{\partial f^{(\ell)}}{\partial S}\right),
\end{equation}
may be a 
practical approximation to Eqs.\ \ref{eq:hamiltons} when 
$C_{\text{transport}}$ is small and observed dynamical allocation of an system resource 
relaxes monotonically to the static optimum.
We will have occasion to use Eq.\ \ref{eq:overdamped-euler-lagrange} in the context of 
inferring the probability $p(x,t)$ in Chapter 2; in fact, it should be noted that a time discretization 
of Eq.\ \ref{eq:overdamped-euler-lagrange} corresponds exactly to 
minimization of Eq. \ref{eq:general-action} via functional gradient descent, written as 
\begin{equation}\label{eq:func-grad-descent}
	S_{n+1}(x) = S_n(x) - \gamma \nabla_S \mathscr{L}_{\text{static}}(S_n(x)),
\end{equation}
where the learning rate $\gamma$ corresponds with the inverse friction $k(x)^{-1}$
and $\mathscr{L}_{\text{static}}(S) = p(x) L(S(x)) + \lambda_{\ell}f^{(\ell)}(S(x))$.

\section{Discovery of underlying distributions}\label{sec:inverse}
We now consider the psuedo-inverse problem to the one discussed above and 
propose an algorithm for its solution.
Suppose we observe a noisy representation of a system resource $Y(x) = S(x) + \eps$
that is {\it prima facie} distributed unequally over some domain.
We wish to find the density distribution $p(x)$ in accordance with which the system resource is 
optimally distributed as outlined above. 
If given a family of candidate distributions $\{p_i(x)\}_1^{n}$ and 
a family of candidate models $\{f_j(x|p)\}_{j=1}^{m}$, we may determine the most likely
underlying distribution as follows:
for each distribution $p_i$ and candidate function $f_j$, fit the model $\hat{Y}_{i,j}(x) = f_j(x|p_i)$.
Let $\mathcal{D}_i$ be the functional defined by the solution to $\partial_t q = \nabla^2 q$ on $\Omega$ 
with $q(x, 0) = p_i(x)$, so that
$p_i \xmapsto{\mathcal{D}_i} \text{Unif}$, the uniform distribution.
Then compute 
\begin{equation}
\hat{Y}^{\text{diffused}}_{i,j}(x) = \mathcal{D}_i \circ f_j(x|p_i)
\end{equation}
and choose $p^*$, the optimal distribution, as 
\begin{equation}\label{eq:diff-machine}
	p^*(x) = \argmin_{p_{i}:\ i\in\Set{1,...,n},\ j\in\Set{1,...,m}} ||\nabla \hat{Y}^{\text{diffused}}_{ij}(x)||,
\end{equation}
for $||\cdot ||$ some appropriate norm.
In other words, given sets of candidate distributions and functions,
the distribution for which the system was most 
likely designed is the one that, when diffused, generates via the transform $\mathcal{D}_i \circ f_j$ 
the most evenly-distributed diffused resource $\hat{Y}^{\text{diffused}}_{i,j}$.
\medskip

\noindent
In application we will notice two hurdles that will affect the utility of this algorithm. 
First, in finite time we know that $\mathcal{D}_i$ will not actually generate the uniform distribution. 
Even if an analytical solution to the diffusion equation is used (e.g., Fourier cosine series, as used here
and in \cite{gastner2006optimal}) one must make a finite approximation. 
Second, and more problematically, there is no general principled way to construct the functions $f_i$ given
only this collection of distributions and an observed resource. 
In practice these functions must be generated either using statistical methods 
or from first principles; one may use this decision process to as a method to help determine which 
physical theory of many under consideration is more likely to be correct.
\medskip

\noindent
Implementation of the above procedure would proceed using standard methods.  
For simplicity's sake, use the $L^1$ norm and approximate as 
\begin{equation}
\begin{aligned}
||\nabla \hat{Y}^{\text{diffused}}_{i,j}(x) || 
&= \int_{\Omega} \dee{x}\ |\nabla \hat{Y}^{\text{diffused}}_{i,j}(x) | \\
	&\approx \sum_{x \in \Lambda(\Omega)} |\nabla_{\text{disc}} \hat{Y}^{\text{diffused}}_{i,j}(x) |,
\end{aligned}
\end{equation}
where $\Lambda(\Omega)$ is a lattice approximation to $\Omega$ and $\nabla_{\text{disc}}$ is a discrete
approximation to the gradient.
Calculating this quantity for each combination $(p_i, f_j)$ and taking the most evenly distributed quantity 
should give the most nearly correct distribution.

\section{Concluding remarks}

We have demonstrated a property that appears to be universal to many physical and social systems, 
summarized as follows: resources that appear to be unevenly distributed in optimized systemsial are, in fact, 
evenly distributed with respect to some other distribution on the underlying space.
This is the conserved quantity $\langle p \rangle + \lambda_{\ell}\frac{\partial f^{(\ell)}}{\partial S} / 
\frac{\partial \pi}{ \partial S}$.
This description is not limited to static systems, as we have extended the framework to allow for time-dependent
allocations {\it v\'is-a-v\'is} transport costs and even for moving coordinate systems.
In constructing this further generalization, we note that the static optimum arises naturally as a special
case.
Finally, we outline the partial inverse problem of determining a distribution $p(x)$ for which an observed quantity 
$Y = S(x) + \eps$ was most likely optimized---assuming that the system is governed by the generalized equipartitioning 
principle.
\medskip

\noindent
We note a meta-optimization procedure that is implied by the existence of the equipartitioned system. 
Let us take the context of machine learning as an example. 
We want to know the true distribution $p(x)$; we are interested in finding $S(x) \in \mathcal{V}$ that minimizes $L$; 
we should analyze the loss function $L$; we should understand the form of the $n$ constraints. 
In order to do this one must consider all factors of the minimization problem:
\begin{itemize}
	\item the probability distribution $p(x)$ (or measure $P(x)$)
	\item the loss function $L$
	\item the functional form of $S$---that, is the function space $\mathcal{V}$ and its characterization
	\item the constraints $f_i$---their functional form and their number
	\item the domain of integration $\Omega$
\end{itemize}
Each one of these components of the optimization can be analyzed and, in a sense, optimized themselves.
\begin{figure}[!htp]
\centering
	\includegraphics[width=\textwidth]{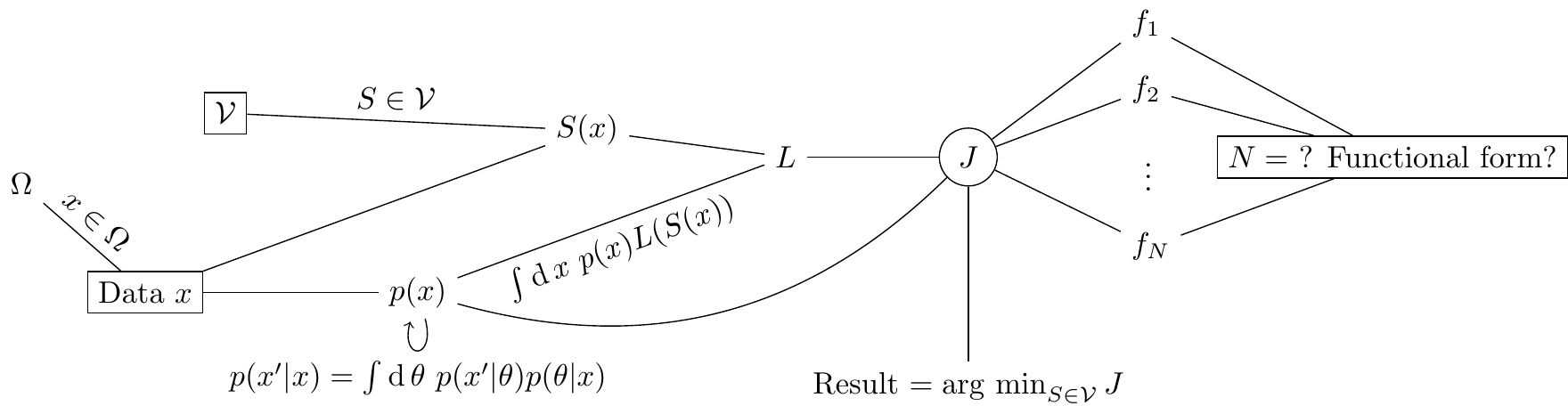}
	\caption{A decomposition of a system subject to the generalized equipartitioning principle into 
	its component parts. A system designer must consider each of these parts carefully when 
	implementing or analyzing such a system.
	In particular, we consider the specification of $p(x)$ and its inference in \\Chapter 2.}
	\label{fig:meta-optimization}
\end{figure}
Figure \ref{fig:meta-optimization} 
demonstrates this meta-optimization process and a decision mechanism for its implementation. 
Every square box is an input that is controlled by the system designer: what data to use in the model, what function 
space to search (linear, quadratic, integrable,...), what form the constraints take and how many of them there are. 
Other nodes must also be considered; while the probability of observing data is a determined quantity once a dataset 
is chosen, how that probability distribution is updated (if at all) can be determined by the system designer.
All of these components feed into the action $J$, which generates the optimal system state.

\chapter{Estimation of governing probability distribution}
\chaptermark{Estimation of governing probability distribution} 
\begin{quote}
When attempting to optimize the state of a system governed by the generalized equipartitioning principle, it is 
vital to understand the nature of the governing probability distribution.
We show that optimiziation for the incorrect probability distribution can have catastrophic results, e.g., infinite 
expected cost, and describe a method for continuous Bayesian update of the posterior predictive distribution 
when it is stationary.
We also introduce and prove convergence properties of a time-dependent nonparametric kernel density estimate (KDE)
for use in predicting distributions over paths.
\end{quote}

\section{Misspecification}\label{sec:misspec}

\subsection{Loss due to misspecification}
Consider two probability density functions $p(x)$ and $q(x)$. 
Suppose we have minimized the functional 
$$
J(p) = \int \left( p(x)L(S(x)) + \sum_{i=1}^M\lambda_i f_i(S(x))\right) \dee x,
$$
but in fact the ``true" functional to minimize is $J(q)$ --- we have mistaken $p(x)$ for the 
true density $q(x)$.
Denoting by $S^{(p)}(x) = \argmin_{S \in \functionspace} J(p; S)$ and $S^{(q)}(x) = \argmin_{S \in \functionspace} J(q; S)$,
the expected opportunity cost due to the misspecification is given by 
\begin{equation}\label{eq:opp-cost}
	\begin{aligned}
		\langle C_{\text{opp}}(p||q)\rangle &= \int \left(q(x) - p(x)\right) L(S^{(p)}(x))\ \dee x\\
		&= \left \langle C^{(p)}\right \rangle_q - \left \langle C^{(p)}\right \rangle_p,
	\end{aligned}
\end{equation}
where we use the notation $\langle C^{(u)} \rangle_v = \int v(x) L(S^{(u)}(x))\ \dee{x}$.
We will see in Sec.\ \ref{sec:misspec-consequences} that misspecification can, especially in unbounded domains, 
lead to rather dramatic consequences.
A useful quantity is the proportion of total cost incurred under the distribution $q$ by misspecifying for the distribution $p$, given by
\begin{equation}
	\begin{aligned}
		\rho(p||q) &= \frac{\langle C_{\text{opp}}(p||q)\rangle}{\left \langle C^{(p)}\right \rangle_q}\\
		&= 1 - \frac{\left \langle C^{(p)} \right \rangle_p}
		{\left \langle C^{(p)}\right \rangle_q }.
	\end{aligned}
\end{equation}
As the opportunity cost becomes the majority of the cost incurred in the system as a whole, 
we have $\rho(p||q) \rightarrow 1$; we will demonstrate an example of this presently.

\subsection{Estimation of $q$}\label{sec:estimation}
How should we estimate the true distribution $q(x)$ as we observe a time-ordered sequence of data during 
the optimization process?
The most obvious answer is to use some sort of Bayesian estimation process, but as we must (in the limit) update 
continuously in time, the method by which this is accomplished is not obvious. 
We will first outline the theory where $q(x)$ is stationary, $t \in [0,1]$, and $x \in [a, b]$; though we have not extended 
the theory to $\R^N$, it should follow directly from consultation with the particle filter literature
\cite{del1996non}.
We then extend the theory to nonstationary $q(x)$ and present a nonparametric estimation procedure that converges to the 
true distribution as time progresses.

\subsubsection{Stationary q(x)}
Updating occurs via a kind of particle filter.
Given an initial prior distribution $p(\theta_0)$ and $k$ observed data points $x_1(t_1),...,x_k(t_k)$ 
ordered so that $t_1< \cdots < t_k$, we compute the posterior 
distribution as
\begin{equation}
	p_k(\theta|x) = p(\theta_k|x_k,...,x_1) = 
	\frac{p(\theta)\likelihood_k(\theta|x)}{p(x_k)},
\end{equation}
where $\likelihood_k(\theta|x) = \prod_{j = 1}^k p(x_j|\theta_j)$ and 
$$
p(x_k) = \prod_{j=1}^k \int p(x_j|\theta_j)p(\theta_j|x_{j-1},...,x_1) d\theta_j.
$$
The posterior predictive distribution is then 
\begin{equation}\label{eq:ppd}
	p(x_{k+1}|x_k,...,x_1) 
	= \int p(x_{k+1} | \theta_k) p(\theta_k | x_k,...,x_1)\ \dee \theta_k.
\end{equation}
As $k \rightarrow \infty$, $p(x_k|x_{k-1},...,x_1) \rightarrow q(x)$; this is thus the correct distribution for the 
decision-maker to use at time $t_k$. 
We must now rectify this procedure with the continuous process of dynamically allocating the system resource 
$S(x,t)$. 
Recall from Eq. \ref{eq:overdamped-euler-lagrange} that we can consider the state of the system resource 
as evolving via gradient descent, 
\begin{equation}
\frac{\partial S}{\partial t} = -\frac{\delta J}{\delta S}.
\end{equation}
 Let $n \in \N$ be given, let $\Set{t_k}_{k=1}^n$ be a partition of $[0,1]$ with 
 $t_k - t_{k-1} = \frac{1}{n}$, and 
$\Set{x_k}_{k=1}^n$ be a partition of $[a,b]$.
From Eq. \ref{eq:func-grad-descent}, the discrete evolution equation is 
\begin{equation}\label{eq:discrete-evol-stationary}
	\begin{aligned}
		S_{k-1}(x) - S_k(x) &= -\frac{1}{n}\frac{\delta J}{\delta S} \\
		&= -\frac{1}{n}\left( p(x|x_{k-1},...,x_1)\frac{\partial L}{\partial S} 
		+\sum_{i=1}^M\lambda_i \frac{\partial f_i}{\partial S}\right), 
	\end{aligned}
\end{equation}
where we use our discrete-update posterior predictive distribution (Eq.\ \ref{eq:ppd}) as the true
probability distribution.
As the above equation is well-defined for any $n \in \N$,
taking $n \rightarrow \infty$ {\it a priori} gives the correct evolution equation for $S$ when the 
estimated probability density updates continuously in time.
We note that these probability distributions may be
indexed by a collection of information sets, $\Set{I_t}_{t\in [0,1]}$ with $I_t \subseteq I_{t+s}$,
such that for any increasing
convergent sequence $\Set{t_k}_{k=1}^{\infty} \subseteq [0,1]$, the collection $\Set{I_{t_k}}_{k=1}^{\infty}$ 
satisfies 
\begin{equation}
	p(x|x_{k-1},...,x_1) = p(x|I_{t_k}).
\end{equation}
Then the evolution equation is given by 
\begin{equation}
	\frac{\partial S}{\partial t} = - \left( p(x|I_t)\frac{\partial L}{\partial S} 
		+\sum_{i=1}^M\lambda_i \frac{\partial f_i}{\partial S} \right).
\end{equation}
\medskip

\noindent
The careful reader will note an implicit assumption in the assertion that 
$\lim_{n \rightarrow \infty} p_n \rightarrow q$.
As the sequence of observations is time ordered---we observe $x_1$, then $x_2$, etc., we are assuming that the 
process generating these samples is ergodic, to be interpreted as follows.
Let $T:[a,b] \rightarrow [a,b]$ be a measure-preserving 
transformation on the support of $q$ that evolves the observation of the data $x$. 
Let $Q$ be a probability measure on $[a,b]$ and let $f$ be a $Q$-measurable function denoting some observable, 
where $dQ/dx = q$.\footnote{To be interpreted as the Radon-Nikodym derivative for 
$Q$ absolutely continuous with respect to Lebesgue measure if $Q$ is nondifferentiable as a real function.} Then
our assumption is
\begin{equation}\label{eq:ergodic}
	\lim_{n \rightarrow \infty}\frac{1}{n}\sum_{k=1}^n f(T^k x_0) = \int_{[a,b]} f\ \dee Q
\end{equation}
almost surely, where $x_0$ is some initial point from which $T$ begins to evolve the observations.
(When the observable is the probability distribution itself we take $f = K_h$ in Eq.\ \ref{eq:ergodic} to
be some proper kernel function.)
If this requirement is not satisfied the above result does not hold.

\subsubsection{Nonstationary $q(x,t)$}\label{sec:nonstationary}
The convergence properties of the above nonparametric estimation largely carry over to the nonstationary case where
the decision-maker observes multiple event paths $\Set{X_k(t)}_{k=1}^N$ and wishes to estimate $q(x,t)$.
We will define the probability kernel to be Gaussian, 
$K_h(x,t) = \frac{1}{2\pi h}\exp(-\frac{1}{2h}(x^2 + t^2))$.
At each $X_k(t) = (x_k, t_k)$, place a kernel $K_h(x-x_k, t-t_k)$ and make the finite estimate
\begin{equation}
	G_h^{(N)}(x.t) = \frac{1}{N}\sum_{k=1}^N K_h(x-x_k, t-t_k).
\end{equation}
An appropriate limit of this function converges to the true distribution $q(x,t)$. 
To see this, note first that 
\begin{equation}\label{eq:lim-kde}
\lim_{N \rightarrow \infty} G_n^{(N)}(x,t) = \int \dee{t'}\ \int \dee{x'}\ K_h(x-x', t-t')q(x',t'),
\end{equation}
by (assumed) ergodicity,
where the integrals are meant in the Riemann sense.
(We will denote $\lim_{N \rightarrow \infty}G_h^{(N)}(x,t) = G_h(x,t)$.)
Define the Fourier transform to be 
$\tilde{F}(\omega) = F[f](\omega) = \int \dee r f(r) e^{\i \omega r}$ and 
note that 
\begin{equation}
	\begin{aligned}
		\tilde{G_h}(\omega) = F[G_h](\omega) &= \tilde{K_h}(\omega)\tilde{q}(\omega)\\
		&=\exp\left( -\i\mu \omega - \omega^T \Sigma \omega \right)\tilde{q}(\omega),
	\end{aligned}
\end{equation}
by the convolution theorem, where $\mu = (x', t')$ and $\Sigma = 
\begin{psmallmatrix}
	h & 0\\
	0 & h
\end{psmallmatrix}
$
is the covariance matrix. 
Now $\tilde{G}(\omega) = \lim_{h \rightarrow 0^+} \tilde{G_h}(\omega) = \exp(-\i \mu \omega)\tilde{q}(\omega)$, 
whereupon the inverse Fourier transform gives
\begin{equation}
	G(x,t) = \int \dee t'\ \int \dee x'\ \delta(x' - x,t'-t) q(x',t') = q(x,t),
\end{equation}
as claimed.
\medskip

\noindent
We emphasize two points. 
First, note that this is {\it not} an algorithm for predictive inference of individual sample paths
$x(t)$ but a nonparametric estimation technique for {\it a posteriori} updating about the distribution
of all sample paths arising from some process.
Second, though in the proof of convergence we set $K_h$ to be Gaussian, this is not strictly necessary. 
This proof holds in the case of an arbitrary kernel $\mathcal{K}_h(x - x', t - t')$ with 
bandwidth function $\Sigma(h)$ that satisfies
$\lim_{h \rightarrow 0^+}F[\mathcal{K}_h](\omega) = \exp(-\i \mu \omega)$, as this is at the core of the proof.

\section{Examples and application}\label{sec:examples}

\subsection{Misspecification consequences}\label{sec:misspec-consequences}

Consider a cost minimization problem
\begin{equation}\label{prob:ez-hot}
	\begin{aligned}
		\min_{S(x)} &\int q(x)C(x)\ \dee x\\
		& \text{s.t.} \int S(x)\ \dee x = K,
	\end{aligned}
\end{equation}
where $C(x) = S(x)^{-1}$, a simple form of the HOT formalism.
The solution to Eq.\ \ref{prob:ez-hot} is found to be
\begin{equation}\label{eq:ez-hot-sol}
S(x) = \frac{Kq(x)^{\frac{1}{2}}}{\int q(x)^{\frac{1}{2}}\dee x}.
\end{equation}
Suppose that we actually optimize for $p$ rather than $q$; we are interested in the quantities
$\langle C_{\text{opp}}^{(p)} \rangle$ and $\rho(p||q)$.
We will consider $p = \mathcal{N}(0, \sigma_p)$ and $q = \mathcal{N}(0, \sigma_q)$ in $\R^2$, where 
$\sigma_q \geq \sigma_p$, and consider the opportunity cost $\langle C_{\text{opp}} \rangle$ 
over $\Omega$ compact and $\R^2$.
Substituting Eq. \ref{eq:ez-hot-sol} into Eq.\ \ref{eq:opp-cost}, we have
\begin{equation}
	\begin{aligned}
		\langle C_{\text{opp}}^{(p)} \rangle 
		&= \int \dee x\ \dee y\ \left(q(x,y)p(x,y)^{-1/2} - p(x,y)^{1/2} \right)\\
		&= \int \dee x\ \dee y\ \Big[(2\pi)^{1/2}
		e^{(\frac{1}{4\sigma^2}- \frac{1}{2\sigma_q^2})(x^2 + y^2)}\\
		&\qquad\qquad- (2\pi)^{-1/2}e^{-\frac{1}{4\sigma^2}(x^2 + y^2)}\Big].
	\end{aligned}
\end{equation}
When this integral is taken over all $\R^2$, it converges only when $\sigma_q^2 < 2 \sigma_p^2$;
$\rho(p||q)$ approaches one as $\sigma_q^2 \rightarrow 2 \sigma_p^2$ when the domain of integration
is unbounded.
Figure \ref{fig:ez-hot-misspec} demonstrates this convergence on a compact domain 
$\Omega \subset [0,1] \times [0,1]$ (displayed in the inset plot) and on $\R^2$.
This emphasizes the importance of accurate estimation of $q$; in the case where $\Omega = \R^2$, we see that 
the opportunity cost quickly becomes the dominant cost term.
\begin{figure}[!htp]
\centering
	\includegraphics[width=\columnwidth]{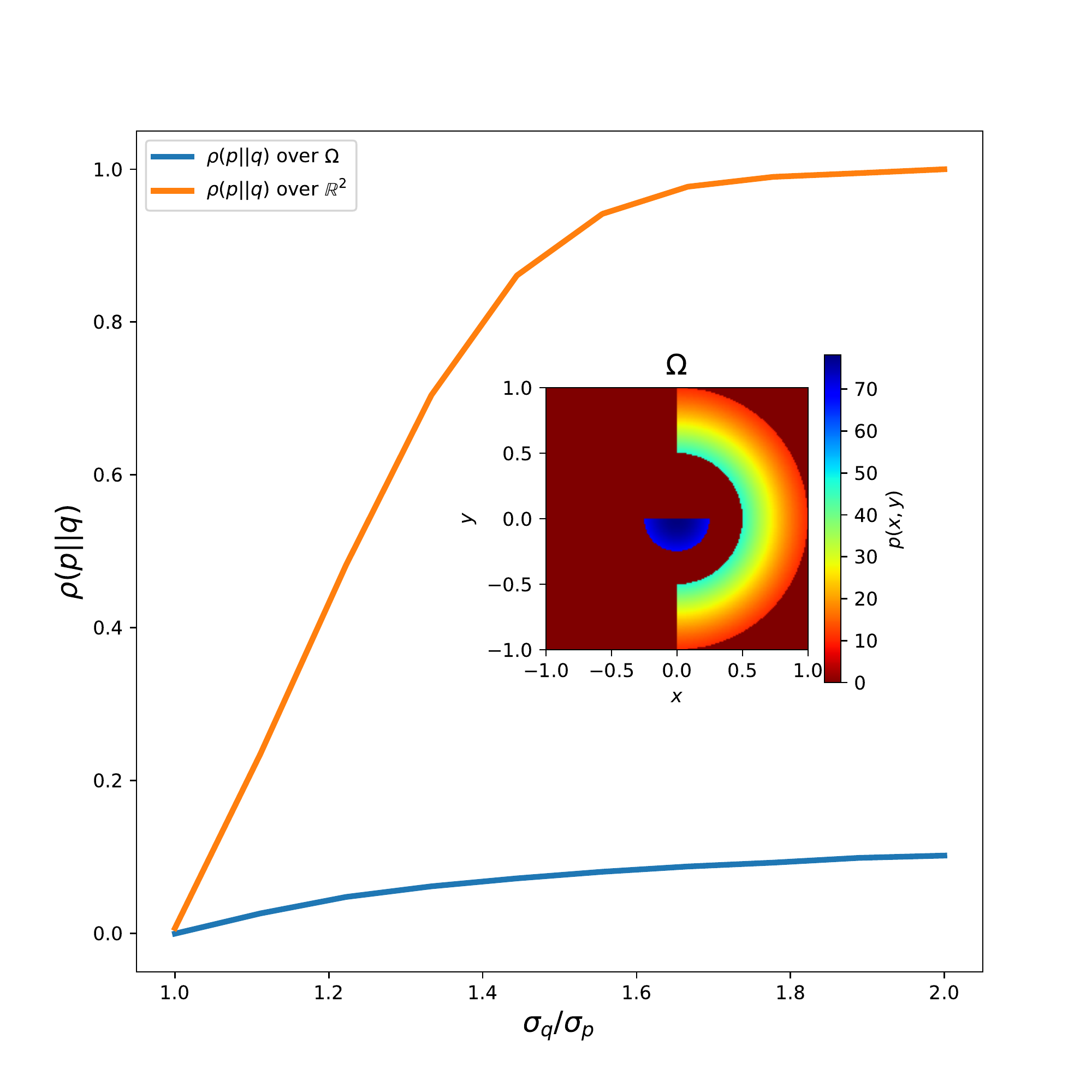}
	\caption{
		Proportion of cost due to opportunity cost in Eq.\ \ref{prob:ez-hot}.
		The probability densities $p$ and $q$ are Gaussian, with $q$'s standard
		deviation ranging from one to twice the size of $p$'s. 
		The integrals always converge on compact $\Omega$; for $\Omega$ small enough
		(in the Lebesgue-measure sense) in proportion to the standard deviation of $q$,
		the proportion converges to a relatively small value as $q$ appears more and 
		more like the uniform distribution. 
		As $\sigma_q / \sigma_p \rightarrow 2$ the integral diverges and $\rho\rightarrow 1$.
		Integrals were calculated using Monte-Carlo methods. 
		(We choose $\Omega$ to be disconnected to emphasize the notion that 
		the generalized equipartitioning principle applies to arbitrary domains.)
		}
		\label{fig:ez-hot-misspec}
\end{figure}

\subsection{Example: discrete allocation}
A practical implementation of problem \ref{prob:ez-hot} in $\R^n$ entails division of $\Omega \subset \R^n$ into 
$M$ compact subdomains to each of which the system resource is allocated; 
we can imagine that we discretize the space and allocate $M$ blobs of resource to each in order to mitigate cost
inside each blob's respective subdomain.
Supposing $x \sim \text{Categorical}(p_1,...,p_M)$ and $p \sim \text{Dirichlet}(\alpha_1,...,\alpha_M)$,
a standard calculation gives the posterior predictive distribution for each new observed $x_N$ to be 
$$
p(x_N|x_{N-1},...,x_1) = \sum_{i=1}^M \delta(x,i)\left(\frac{\alpha_i + n_i}{N + \sum_{i=1}^M \alpha_i}\right),
$$
where $\delta(x,i)$ is one if $x=i$ and zero otherwise.
Using the above analytical solution, the empirical optima are given by 
\begin{equation}
	\hat{S}(x,t_N) = \frac{K\left(\sum_{i=1}^M \delta(x,i)
	\left(\frac{\alpha_i + n_i}{N + \sum_{i=1}^M \alpha_i}\right)\right)^{\frac{1}{2}}}
	{\sum_{i=1}^M \left(\frac{\alpha_i + n_i}{N + \sum_{i=1}^M \alpha_i}\right)^{\frac{1}{2}}}.
\end{equation}
In taking the continuum limit $N \rightarrow \infty$, it is tacitly assumed that all $p_i \in (0,1)$; that is, 
we must have $n_i = \Theta(N)$ for all $i$.
The dynamic allocation of $S$ for any instance of this problem is given by the solution to the differential equation 
\begin{equation}\label{eq:ez-hot-dynalloc}
	\frac{\partial S(x,t)}{\partial t} = p(x|I_t)S(x,t)^{-2} - 
	\left( \frac{1}{K}\int p(x|I_t)^{\frac{1}{2}}\ 
	\dee x\right)^2,
\end{equation}
which here it takes a simplified form due to the discrete nature of the spatial coordinate.
Figure \ref{fig:ez-hot-categorical} demonstrates the convergence of the dynamic allocation given by 
Eq.\ \ref{eq:ez-hot-dynalloc} to the static optima as $p_k(x) \rightarrow q(x)$
as the probability is updated according to the procedure outlined in Sec.\ \ref{sec:estimation}.

\begin{figure}[!htp]
\centering
	\includegraphics[width=\columnwidth]{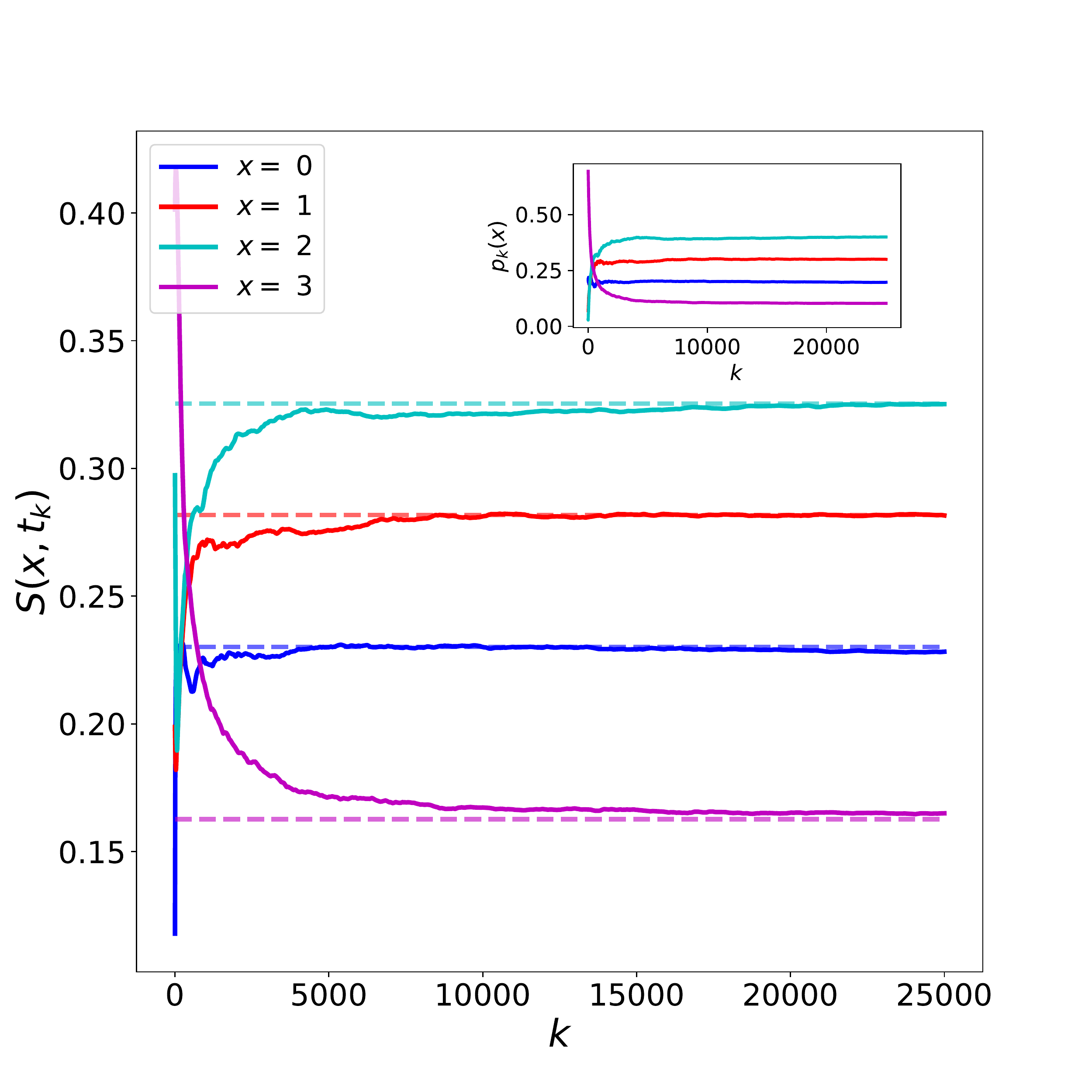}
	\caption{
		Dynamic allocation of system resource in the toy HOT problem given by Eq.\ \ref{prob:ez-hot}.
		Dashed lines are the static optima when the true distribution $q(x)$ is known. 
		The solid lines are the dynamic allocation of $S(x,t)$ as the estimate $p_k(x)$ is updated.
		The inset plot illustrates the convergence of $p_k(x)$ to the true distribution via the updating
		process described in Sec.\ \ref{sec:estimation}.
		To demonstrate the effectiveness and convergence properties of the procedure we initialize the 
		probability estimates and initial system resource allocations to wildly inaccurate values.
		}
		\label{fig:ez-hot-categorical}
\end{figure}

\subsection{Example: continuous time update with nonstationary distribution}
We briefly mentioned an example of a common nonstationary distribution in Section \ref{sec:nonstationary}; 
we continue this discussion now.
Consider the probability distribution on $\R$ induced by the Weiner process 
$\dee{X_t} = \mu\ \dee{t} + \sigma\ \dee{W_t}$ with $X_0 = x_0$.
Intuitively the stationary update process could not hope to produce a realistic estimate of the true probability
distribution $q(x,t)$, which is given by
\begin{equation}\label{eq:weiner-dist}
	q(x,t) = \frac{H(t)}{\sqrt{2\pi \sigma^2 t}}\exp \left( \frac{-(x - \mu t)^2}{2 \sigma^2 t} \right),
\end{equation}
where $H(t)$ is Heaviside's function\footnote{See Appendix A for a derivation.}. 
We implement the nonstationary updating procedure described in Sec.\ \ref{sec:nonstationary} and infer the probability 
distribution; results are displayed in Figure \ref{fig:weiner}.
The top plot displays the results of the estimation, while the bottom displays results of the true unconditional
distribution given by Eq.\ \ref{eq:weiner-dist}.
\begin{figure*}
\centering
	\includegraphics[width=\textwidth]{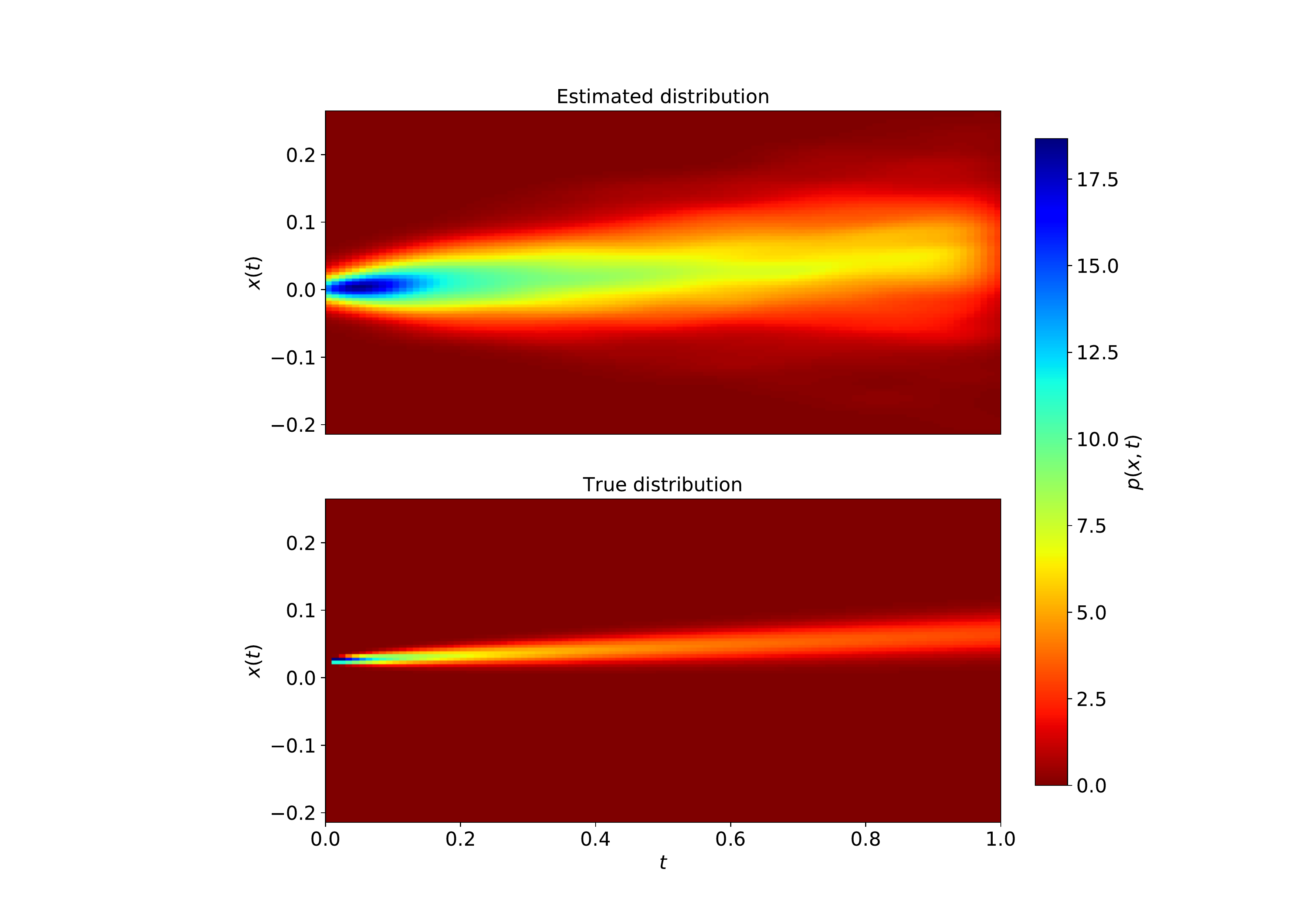}
	\caption{
		Empirical estimation of the distribution $q(x,t) \sim \mathcal{N}(\mu t, \sigma \sqrt{t})$ 
		generated by a Weiner process with drift $\mu$ and volatility $\sigma$.
		The estimation was generated using the procedure described in Sec.\ \ref{sec:nonstationary}.
		}
	\label{fig:weiner}
\end{figure*}

\chapter{Equipartitioning on networks}
\chaptermark{Equipartitioning on networks} 
\begin{quote}
As noted in the introduction, even over non-continuum domains the formalism of the generalized equipartitioning 
principle can be used to great effect. 
Here we extend the theory to the case of networks, in which an event probability density is defined over nodes and 
edges and a system resource is to be partitioning among the nodes and edges as well. 
	We derive the governing equations of such a system, showing that they correspond exactly with Eq.\ 
	\ref{eq:static-eq}, and give an example of their application by considering a model 
	of risk propagation on a power grid. 
	We close by identifying an extension of the power grid optimization problem to a more realistic contagion 
	process. 
\end{quote}

\section{Theory}
We consider time-dependent loss functions that account for both node and edge effects. 
Events occur at edge $a_{i,j}$ with probability $p_{i,j}$, though it will be seen that this formulation
can also account for events occuring exclusively in node space.
A network is more intrinsically detailed than the continuum; we must consider the case of allocation of resource 
$S$ to node $i$,
which we will denote by $S_i$, and allocating a (possibly different!) resource $T$ to edge $a_{ij}$, 
denoted by $T_{ij}$. 
The Lagrangian density is the probability density-weighted loss function, 
\begin{equation}
	\mathscr{L} = p_{i,j} L\left(S_i, S_{i,t}, S_j, S_{j,t}, T_{ij,t}, T_{ij,t} \right),
\end{equation}
where we denote the partial derivative of $g$ with respect to $x$ by $g_{,x}$.
The quantity $\int \sum_{i,j}\mathscr{L}\dee t$ is to be minimized subject to $m$ constraints of the form
\begin{equation}
	\int \sum_{i,j}f^{(\ell)}(S_i, S_{i,t}, S_j, S_{j,t}, T_{ij,t}, T_{ij,t} )\dee t = K_{\ell},
\end{equation}
where $\ell = 1,...,m$.
This results in the action given by
\begin{equation}
	J = \int \sum_{i,j}\mathscr{L}\dee t - 
	\sum_{\ell = 1}^m \lambda_{\ell}\left(K_{\ell} - \int \sum_{i,j} f^{(\ell)} \dee t\right)
\end{equation}
(In all generality, $p_{i,j}$ may also evolve in time, but we assume this evolution is governed
by a separate process.)
The optimal intertemporal allocation of resources is given by the Euler-Lagrange equations, which for $S_j$ 
read
\begin{equation}
	p_{ij}\left(\frac{\partial L}{\partial S_j} 
	- \frac{\partial}{\partial t} \frac{\partial L}{\partial S_{j,t}}\right)
	 + \sum_{\ell = 1}^m\lambda_{\ell}\left( \frac{\partial f^{(\ell)}}{\partial S_j}
	 - \frac{\partial }{\partial t} \frac{\partial f^{(\ell)}}{\partial S_{j,t}}\right) = 0.
\end{equation}
The form of the equations is identical for $S_i$ and $T_{ij}$.
From this it can be seen that, in all generality, the governing equations of such an optimization procedure 
are significantly more intimidating than those defined in the continuum, since we are now confronted with a 
system of coupled nonlinear PDEs.
We also note that, unlike in the continuum, the optimal allocation of resources here will depend integrally on 
the nature of the contagion mechanism within the network.
\section{Examples}

\subsection{HOT on networks: node allocation}
We re-do the theoretical analysis of highly-optimized tolerance (HOT) systems on networks, where the spreading mechanism 
considered here is generated by an event at $a_{ij}$ and annihilates both $i$ and $j$ before ceasing to propagate.
We wish to find the optimal allocation of resource at $i$, $S(i)$; we will not consider a resource
allocated to edges. 
The objective function is $\sum_{i,j}p_{ij}C_{ij}$, where we suppose $C_{ij} \propto S_i^{-\gamma_i}S_j^{-\gamma_j}$,
which we minimize subject to the constraint $\sum_i S_i = K$.
The action is 
\begin{equation}
	\sum_{i,j}p_{ij}S_i^{-\gamma_i}S_j^{-\gamma_j}A_{ij} - \lambda\left( K - \sum_i S_i\right).
\end{equation}
Differentiating and rearranging terms gives the scaling relationship
\begin{equation}\label{eq:hot-scaling-1}
	S_i^{\gamma_i + 1} \propto \sum_j p_{ij}S_j^{-\gamma_j}A_{ij}.
\end{equation}
Denoting the conditional probability of an event at $i$ given an event at $j$ by $p_{i|j}$,
we note
$$
\begin{aligned}
	\sum_{j}p_{ij}S_j^{-\gamma_j} &= \sum_j p_{i|j} S_j^{-\gamma_j}A_{ij}p_j \\
	&=\left \langle p_{i|j}S_j^{-\gamma_j} \right \rangle
	_{j \in \mathscr{N}(i)}, 
\end{aligned}
$$
where $\mathscr{N}(i)$ is the neighbor set of node $i$, so that Eq.\ \ref{eq:hot-scaling-1} becomes
\begin{equation}
	S_i \propto \left \langle p_{i|j}S_j^{-\gamma_j} \right \rangle^{\frac{1}{\gamma_i + 1}}
	_{j \in \mathscr{N}(i)}.
\end{equation}
The optimal allocation of resources across the network is given by the simultaneous solution of all $N$ (one
for each node) of these equations. 

\subsection{US power grid: edge allocation}

As a practical example we consider the minimization of cost in a power grid. 
Suppose that events occur at nodes (power generation facilities, substations, etc.) that impose costs on 
the rest of the network via propagation along edges (transmission lines). 
A resource $S_{ij}$ may be deployed on transmission lines to alleviate these costs (for simplicity we will
assume $S_{ij} \propto C_{ij}^{-1}$) that also imposes a (monetary) cost on the transmission of 
electricity; we wish to minimize the expected event cost subject to the constraint that the total monetary
cost throughout the network is within our budget.

\subsubsection{Neighborhood costs}
In the simplest case, events at node $i$ affect only $i$'s neighbor nodes and monetary cost 
scales linearly with resource placement.
Define $A$ to be the adjacency matrix of the power grid and assume that edges are undirected so that $A$ is symmetric.
The form of the action is then
\begin{equation}\label{eq:power-grid-neighborhood}
\begin{aligned}
	J &= \sum_i p_i \sum_{j \in \mathscr{N}(i)} S_{ij}^{-1} + \lambda \left(K - \sum_{ij}S_{ij}  \right) \\
	&= \sum_i p_i \sum_j S_{ij}^{-1}A_{ij} + \lambda \left(K - \sum_{ij}S_{ij}  \right)\\
	&= p_i S_{ij}^{-1}A_{ij} + p_j S_{ji}^{-1}A_{ji} + \cdots \text{other terms}\\
	&= \underbrace{(p_i + p_j)S_{ij}^{-1}A_{ij}}_{\text{since the network is undirected}} 
	+ \cdots \text{other terms}.
\end{aligned}
\end{equation}
Performing the optimization gives the optimal scaling of resource $S_{ij}$ as 
$S_{ij}^2 = \lambda^{-1}(p_i + p_j) A_{ij}$.
Saturation of the resource constraint yields $\lambda = K^{-2} 
\left( \sum_{i,j} (p_i + p_j)^{1/2}A_{ij}\right)^2$, 
whereupon substitution into the scaling relationship gives the analytical optimum to be 
\begin{equation}
	S_{ij} = \frac{K (p_i + p_j)^{1/2}A_{ij}}{\sum_{k, \ell }(p_k + p_{\ell})^{1/2}A_{k\ell}}.
\end{equation}
As noted above, these networked optimization problems depend heavily on the underlying contagion mechanism,
which can introduce its own distributional effects. 
In the problem considered above, the neighborhood contagion mechanism generates direct dependence of 
the allocation of the resource on node degree, which was not explicitly considered
in the problem formulation.
If the event probability distribution $p_i$ is not correlated with neighborhood structure, it is 
in fact possible under this contagion mechanism for the associated distribution $\Pr(p_i + p_j)$ to be nearly 
constant over the entire network. 
Consider the case of an infinite network with contagion mechanism and event distribution as 
defined above.
Since $\sum_i p_i = 1$, for any $\eps >0$ there exists $N \in \N$ such that for all $n \geq N$,
$ p_n < \eps$ \footnote{This is an incredibly weak statement.}.
Choose $i$ any node and let $j$ be its neighbor. 
Since the probability distribution is uncorrelated with neighborhood structure,
the pair $p_i$ and $p_j$ are statistically identical to any arbitrary pair of node event probabilities
$p_k$ and $p_{\ell}$. 
From the above, $p_k + p_{\ell} \leq 2 \eps$ for most pairs $k$ and $\ell$; 
we thus expect the associated distribution $\Pr(p_i + p_j)$ to be tightly centered around a small 
value.
Thus in the case of uncorrelation of event probability with neighborhood structure 
we can approximate $S_{ij} \approx \frac{KA_{ij}}{\sum_{k,\ell}A_{k,\ell}}$, so that
\begin{equation}\label{eq:power-grid-nhbd-approx}
\sum_j S_{ij} \approx \frac{K \rho_i}{\sum_k\rho_k},
\end{equation}
where $\rho_i$ is the node degree of $i$.
Figure \ref{fig:us-power-grid-neighborhood} displays the
theoretical optimum in Eq.\ \ref{eq:power-grid-nhbd-approx} 
along with results from simulation on the western US power grid dataset\footnote{
Data available at \url{http://konect.uni-koblenz.de/networks/opsahl-powergrid}}.
The event probability $p_i$ was set without dependence on $i$'s node degree; the inset plot notes that the
distribution of $p_i + p_j$ is nearly constant as predicted above. 
We thus treat $p_i + p_j \approx \text{constant}$ and 
plot the resulting linear fit between $\sum_j S_{ij}$ and $K\rho_i / \sum_{k} \rho_{k}$ in the main plot,
which demonstrates good agreement between simulation and theory.

\begin{figure}
\centering
	\includegraphics[width=\textwidth]{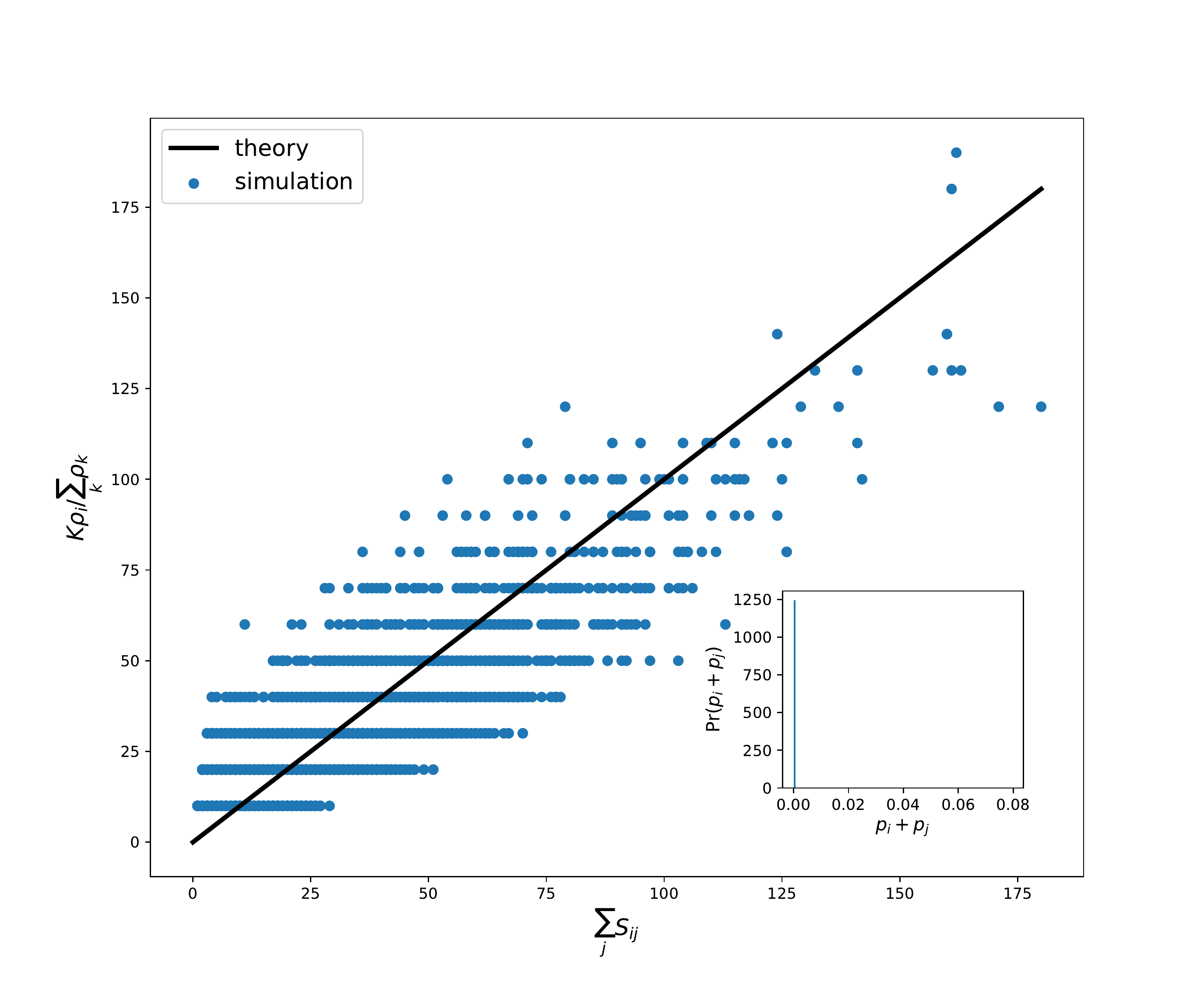}

	\caption{
		Simulated optimum of Eq.\ \ref{eq:power-grid-neighborhood}
		plotted against the theoretical approximate optimum 
		Eq.\ \ref{eq:power-grid-nhbd-approx} on the western US power grid dataset. 
		Eq.\ \ref{eq:power-grid-neighborhood} was minimized using simulated annealing, 
		the implementation of which is described in Appendix B. 
		Optimization was performed with the restriction $S_{ij} \in [1, \infty)$.
		The inset plot demonstrates that $\Pr(p_i + p_j)$ is highly centralized.
		}
	\label{fig:us-power-grid-neighborhood}
\end{figure}

\subsubsection{Subgraph costs with signal loss}
In a more realistic scenario, cost propagates across the subgraph connected to node $i$. 
If we assume slightly lossy transmission lines, signal drop across a path from $i$ to $j$ scales
approximately as $\exp\left( -d(i, j) \right)$, where we will take $d(i,j)$ to be the shortest path distance
between $i$ and $j$ \cite{miano2001transmission}. 
Assuming that event cost scales with signal strength and maintaining the linear 
monetary cost as above, we arrive at the action
\begin{equation}
	J = \sum_i p_i L\left(\text{Subgraph}(i)\right) - \lambda \left( K - \sum_{i,j}S_{ij} \right),
\end{equation}
where we define 
\begin{equation}\label{eq:subgraph-loss}
	L\left(\text{Subgraph}(i) \right) = \sum_{j:\ \text{path}(i,j) \text{ exists}}\
	\sum_{\substack{k \in \text{SP}(i,j) \\ k \succ k'}} e^{-d(i,k')}S^{-1}_{k'k}.
\end{equation}
Here $\succ$ is an ordering on a path such that $k \succ k'$ if $k$ appears after $k'$ in traversing 
the path from $i$ to $j$, and $\text{SP}(i,j)$ denotes the shortest path from $i$ to $j$.
Define $\mathscr{P}_i:G \times G \rightarrow \Z^{\geq 0}$ to be $\mathscr{P}_i(k',k) = \text{\# of times 
$a_{k',k}$ appears in a shortest path in $i$'s subgraph}$.
Then Eq.\ \ref{eq:subgraph-loss} can be rewritten 
\begin{equation}
	L\left(\text{Subgraph}(i) \right) = \sum_{k,k'} \mathscr{P}_i(k',k) e^{-d(i,k')}S_{k',k}^{-1}.
\end{equation}
Extension of this present research could focus on simulating the above problem and comparing the results with 
actual costs in power grids, e.g., damage caused by outages.

\appendix
\addappheadtotoc

\titleformat{\chapter}[hang] 
{\normalfont\huge\bfseries}{\chaptertitlename\ \thechapter:}{1em}{} 

\chapter{Derivations}
\section{Field equations under dynamic coordinates}
We derive the representation of Hamilton's field equations in the case of moving coordinates.
Recall from standard field theory (under stationary coordinates) 
for $S$ a field over $x \in \R^N$ that the Lagrangian density is
given by 
\begin{equation}
\mathscr{L} = T(S,\dot{S}, \nabla S, x, t) - V(S, \dot{S}, \nabla S, x, t),
\end{equation}
with the corresponding action integral
\begin{equation}
	J = \int \dee{x} \int \dee{t}\ \mathscr{L}.
\end{equation}
Defining the conjugate momentum as $\Pi = \frac{\delta \mathscr{L}}{\delta \dot{S}}$, the Hamiltonian is given by
\begin{equation}\label{eq:app1:Hamiltonian}
	\mathscr{H} = \frac{\partial S}{\partial t}\Pi - \mathscr{L},
\end{equation}
with Hamilton's equations thus given as
\begin{align}\label{eq:app1:hamiltons-equations}
	\frac{\partial \Pi}{\partial t} &= -\frac{\delta \mathscr{H}}{\delta S} \\
	\frac{\partial S}{\partial t} &= \frac{\delta \mathscr{H}}{\delta \Pi}.
\end{align}
We claim that this formalism generalizes exactly as one would expect when the Eulerian operator 
$\frac{\partial}{\partial t}$ is replaced by the Lagrangian operator $\frac{D}{Dt} = \frac{\partial}{\partial t}
+ \frac{dx_i}{dt}\frac{\partial}{\partial x_i}$, where we are employing the Einstein summation convention.
Specifically, we claim that 
\begin{theorem}
Hamilton's equations derived from the Hamiltonian 
\begin{equation}
	\mathscr{H} = \frac{DS}{Dt}\Pi - \mathscr{L}
\end{equation}
	are equivalent to the Euler-Lagrange equations derived from the Lagrangian given in Eq. 
	\ref{eq:lagrangian-dynamic} with 
$\frac{\partial}{\partial t} \mapsto \frac{D}{Dt}$.
\end{theorem}

We first show that
\begin{lemma}
	The following holds: $\frac{\delta}{\delta S}\frac{1}{2}\left(\frac{DS}{Dt}\right)^2
	= -\frac{D}{Dt}\frac{DS}{Dt} 
\equiv -\frac{D^2S}{Dt^2}$.
\end{lemma}
\begin{proof}
Computing directly, we have
\begin{equation*}
	\begin{aligned}
		\frac{\delta}{\delta S}\frac{DS}{Dt} &=  -\left( \frac{\partial}{\partial t}
		\frac{\partial}{\partial \dot{S}} + \frac{\partial }{\partial x_i}
		\frac{\partial}{\partial S_{,x_i}}  \right)
		\frac{1}{2}\left( \frac{\partial S}{\partial t} + \frac{dx_i}{dt}
		\frac{\partial S}{\partial x_i} \right)^2\\
		&= -\left(\frac{\partial}{\partial t} + \frac{dx_i}{dt}\frac{\partial}{\partial x_i}  \right)
		\left( \frac{\partial S}{\partial t} + \frac{dx_i}{dt}
		\frac{\partial S}{\partial x_i} \right) \\
		&= -\frac{D}{Dt}\frac{DS}{Dt} = \frac{D^2S}{Dt^2},
	\end{aligned}
\end{equation*}
as claimed.
\end{proof}
We next show that 
\begin{lemma}
The Euler-Lagrange equations for the action 
$J = \int \dee{x}\ \int \dee{t}\ \frac{1}{\alpha}\left( \frac{DS}{Dt}\right)^{\alpha}$ is given by
$-(\alpha - 1)\left( \frac{DS}{Dt} \right)^{\alpha - 2}\frac{D^2S}{Dt^2} = 0$, 
in perfect analogy with the field and particle cases.
\end{lemma}
\begin{proof}
To this end, we compute $\delta J = 0$ and find
\begin{equation*}
	\begin{aligned}
		\frac{\delta}{\delta S}\frac{1}{\alpha}\left( \frac{DS}{Dt}\right)^{\alpha}
		&= \frac{1}{\alpha}\left( \frac{\partial}{\partial S} 
		- \frac{\partial }{\partial t}\frac{\partial}{\partial \dot{S}} - \frac{\partial}{\partial x_i}
		\frac{\partial}{\partial S_{,x_i}}\right)\left( \frac{\partial S}{\partial t} + \frac{dx_i}{dt}
		\frac{\partial S}{\partial x_i}\right)^{\alpha}\\
		&= -\left(\frac{\partial}{\partial t} - \frac{dx_i}{dt}\frac{\partial}{\partial x_i}\right)
\left( \frac{\partial S}{\partial t} + \frac{dx_i}{dt}
		\frac{\partial S}{\partial x_i}\right)^{\alpha-1}\\
		&= -(\alpha - 1)\left( \frac{\partial S}{\partial t} + \frac{dx_i}{dt}
		\frac{\partial S}{\partial x_i}\right)^{\alpha-2}
			\left(\frac{\partial}{\partial t} +\frac{dx_i}{dt}\frac{\partial}{\partial x_i}\right)
		\left( \frac{\partial S}{\partial t} + \frac{dx_i}{dt}
		\frac{\partial S}{\partial x_i}\ \right)
			\\
		&= -(\alpha - 1)\left(\frac{DS}{Dt}\right)^{\alpha -2}\frac{D^2S}{Dt^2},
	\end{aligned}
\end{equation*}
by the definition of the material derivative and the above derivation for the second material derivative.
\end{proof}
We can now prove the theorem.
\begin{proof}
Calculation of the conjugate momentum proceeds in the standard manner,
resulting in $\Pi \equiv \frac{\delta \mathscr{L}}{\delta \dot{S}} = \left( \frac{\partial S}{\partial t} + \frac{dx_i}{dt}\frac{\partial S}{\partial_{x_i}} \right)^{\alpha - 1} \times 1$. 
Forming the Hamiltonian in accordance with Eq.\ \ref{eq:app1:Hamiltonian} results in 
\begin{equation}
	\begin{aligned}
		\mathscr{H} &= \frac{DS}{Dt}\Pi - \mathscr{L} \\
		&= \frac{\alpha - 1}{\alpha}\left( \frac{DS}{Dt} \right)^{\alpha} 
		+ p(x)L(S(x)) + \lambda_{\ell}f^{(\ell)}(S(x)) \\
		&= \frac{\alpha - 1}{\alpha} \Pi^{\frac{\alpha}{\alpha - 1}}
		+ p(x)L(S(x)) + \lambda_{\ell}f^{(\ell)}(S(x)).
	\end{aligned}
\end{equation}
Hamilton's equations are given by Eqs.\ \ref{eq:app1:hamiltons-equations} and take the form
\begin{align}
	\frac{D\Pi}{Dt} &= -p(x)\frac{\partial L}{\partial S} 
	- \lambda_{\ell}\frac{\partial f^{(\ell)}}{\partial S} \\
	\frac{DS}{Dt} &= \Pi^{\frac{1}{\alpha -1}}.
\end{align}
Noting that $\frac{D\Pi}{Dt} = \frac{D}{Dt}\left(\frac{DS}{Dt}\right)^{\alpha - 1}
= (\alpha - 1)\left( \frac{DS}{Dt} \right)^{\alpha - 2}\frac{D^2S}{Dt^2}$ by above results, 
we find that the first of Hamilton's equations is identical to the Euler-Lagrange equation, which 
was the desired result.
\end{proof}

\section{Weiner process probability distribution}
This is essentially a standard derivation that we repeat and elucidate here for completeness's sake. 
We consider the overdamped Langevin equation $\dee{X_t} = \mu\ \dee{t} + \sigma \dee{W_t}$ 
with initial condition $X_0 = x_0$, where this equation is understood in the It\^{o} sense.
By convention we denote the Weiner process by $W_t$. We have $\sigma > 0 $ and 
are interested in deriving the probability distribution $q(x,t)$ of finding a realization of the process
at $x$ at time $t$.
By It\^{o}'s lemma and integration by parts, the Fokker-Planck equation that governs the evolution of $q$ 
on $\R$ is given by 
\begin{equation}\label{eq:qdist-fpe}
	\frac{\partial q}{\partial t} = -\mu \frac{\partial q}{\partial x} + \frac{\sigma^2}{2}
	\frac{\partial^2 q}{\partial x^2},\quad q(x,0) = q_0(x).
\end{equation}
As in Section \ref{sec:misspec}, we will take the initial condition to be $q(x,t) = \delta(x)$.
We solve Eq.\ \ref{eq:qdist-fpe} by means of the Fourier transform, which we will define here as 
$F(\xi) = \frac{1}{\sqrt{2\pi}}\int_{\R} f(x) e^{\i \xi x}\dee{x}$.
Transforming both sides of the equation and the initial condition, we form the ODE 
\begin{equation}
	\frac{\dee{F}}{\dee{t}} = -\left( \mu \i \xi + \frac{\sigma^2}{2}\xi^2 \right)F(t),\quad F(0) = 1,
\end{equation}
the solution to which is given by 
\begin{equation}
	F(t) = \exp\left[-\left(\mu \i \xi + \sigma^2 \xi^2/2  \right) t  \right].
\end{equation}
Setting $\zeta = \mu t$ and $\nu = \sigma\sqrt{t}$, we recognize $F(t)$ as the characteristic function 
of a Gaussian distribution with mean $\zeta$ and standard deviation $\nu$.
Thus $q(x,t)$ is given by 
\begin{equation}
	\begin{aligned}
		q(x,t) &= \frac{1}{\sqrt{2 \pi \nu^2}}\exp\left(\frac{-(x - \zeta)^2}{2\nu^2}  \right) \\
		&= \frac{H(t)}{\sqrt{2\pi \sigma^2 t}}\exp\left( \frac{-(x - \mu t)^2}{2 \sigma^2 t} \right), 
	\end{aligned}
\end{equation}
as claimed above.

\chapter{Software}

\section{Simulated annealing}

Simulated annealing is a Markov Chain Monte Carlo (MCMC) algorithm closely related to the 
celebrated Metropolis-Hastings algorithm. 
We describe it briefly here and detail our software implementation.
\medskip

\noindent
Consider a canonical ensemble exchanging energy, but not particles, with an external heat bath. 
The probability of such an ensemble being in a particular energy state $E$ is given by $
\Pr(E) = \frac{1}{Z} \exp\left( - \beta E \right)$, where we have set Boltzmann's constant to unity in the 
appropriate units, $\beta$ is the inverse temperature of the ensemble, and 
$Z = \sum_{E'} \exp\left(-\beta E' \right)$ is the partition function.
Thus the maximum probability state is that with lowest energy; if the system is such that energy 
in a state $x$ is 
given by the Hamiltonian $\mathcal{H}(x) = E_x$, one may (in principle) find the system configuration 
$x$ that minimizes the system's energy. 
Simulated annealing uses this fact to perform stochastic global optimization.
Algorithm \ref{alg:simulated-annealing} displays the algorithm.
Unlike the standard implementation of simulated annealing in scientific Python, our implementation 
does not assume any underlying set or space in which states $x$ are required to lie;
our implementation can find states that minimize arbitrary Hamiltonians defined 
over elements in any set.\footnote{The standard
implementation can be found at 
\url{https://docs.scipy.org/doc/scipy-0.18.1/reference/generated/scipy.optimize.basinhopping.html}}

\begin{algorithm}
	\caption{
		The simulated annealing algorithm. 
		The function $a$ is a perturbation function that slightly modifies the state $x$ to a 
		``nearby" state $x'$. 
		$P$ is a probability measure on states (in physical scenarios proportional to 
		$\exp\left( - \beta E \right)$), $\eps$ is a numerical tolerance, $B$ is a function that 
		yields successive inverse temperatures, and $\tau$ is a time delay against 
		which to check numerical tolerance.
		}
     \begin{algorithmic}[1]
	     \Procedure{SimulatedAnnealing}{$\mathcal{H}$, $a$, $P$, $\eps$, $B$, $\beta_0$, $x_0$, $\tau$}
	     \State $t \gets 0$
	     \State $x \gets x_0$
	     \State $\beta(t) \gets \beta_0$
	     \While{$\beta < \infty$ and $|E(t + \tau) - E(t)| > \eps$}
		\State $E(t) \gets \mathcal{H}(x)$
		\State $x' \gets a(x)$
		\State $E'(t) \gets \mathcal{H}(x')$
		\State $u \sim \mathcal{U}(0,1)$
		\If{$E' < E$ or $P(E', \beta(t)) / P(E, \beta(t)) \geq u$}
			\State $x \gets x'$
		\EndIf
		\State $t \gets t + 1$
		\State $\beta(t) \gets B(\beta(t-1))$
	     \EndWhile
	\EndProcedure
    \end{algorithmic}
    \label{alg:simulated-annealing}
\end{algorithm}

Clearly the construction of the perturbation function $a$ is critical to the effectivness of this algorithm.
This is a domain-specific question; we will focus here on the cases where $x \in \R^n$ 
or $x \in M^{m \times n}(R)$, the space of $m \times n$ matrices over the ring (or monoid) $R$.
\begin{itemize}
	\item When $x \in \R^n$, we select $k \sim \mathcal{U}_{\text{discrete}}(0, n_{\max})$
		elements of $x$ for perturbation, where $n_{\max} \leq n$.
	We then set $x_i \gets x_i + \xi$ for each selected $x_i$, where
		$\xi \sim \mathcal{N}(0, \sigma(x))$ and $\sigma(x)$ 
		is the standard deviation of the elements of $x$.
	Successive applications of $a$ thus define a type of normal random walk on $\R^n$; 
		this is similar to the original Metropolis jump kernel.
	\item When $x \in M^{m \times n}(\R)$ the perturbation is essentially identical to that 
		outlined above (selecting $k$ random elements of the matrix instead of the vector).
	When $R = \Z$ or $R = \Set{0,1}$ we must alter the algorithm so that 
		$\forall x, y, z \in R$, $x \mapsto yx + z \in R$ as well.
	This is accomplished simply by choosing
		an appropriate probability measure $P$ over $R$, drawing from 
		this distribution $p \sim P$ and generating $x_{ij} \gets x_{ij} + \sigma(x) p$.
		In the particularly simple (and useful!) case where $R = \Set{0,1}$, 
		the initial random selection of matrix elements acts is the only randomization used 
		and the selected elements are just bit-flipped.
\end{itemize}
In some cases we may want to restrict elements to certain subsets of $R$. 
This is accomplished by checking whether the new point $x'$ is in the desired subset $\Sigma$
(which we take to be a compact set); if it is not, 
$x'$ is assigned to be $x'' = \argmin_{x'' \in \partial \Sigma} ||x' - x''||_2$,
where $\partial \Sigma$ is the boundary of $\Sigma$.
Figure \ref{fig:sim-ann-conv} demonstrates our implementation of the simulated annealing algorithm converging
to the global minimum of Eq. \ref{eq:power-grid-neighborhood} with the restriction that $x_i \in [1, \infty)$.

\begin{figure}
\centering
	\includegraphics[width=\textwidth]{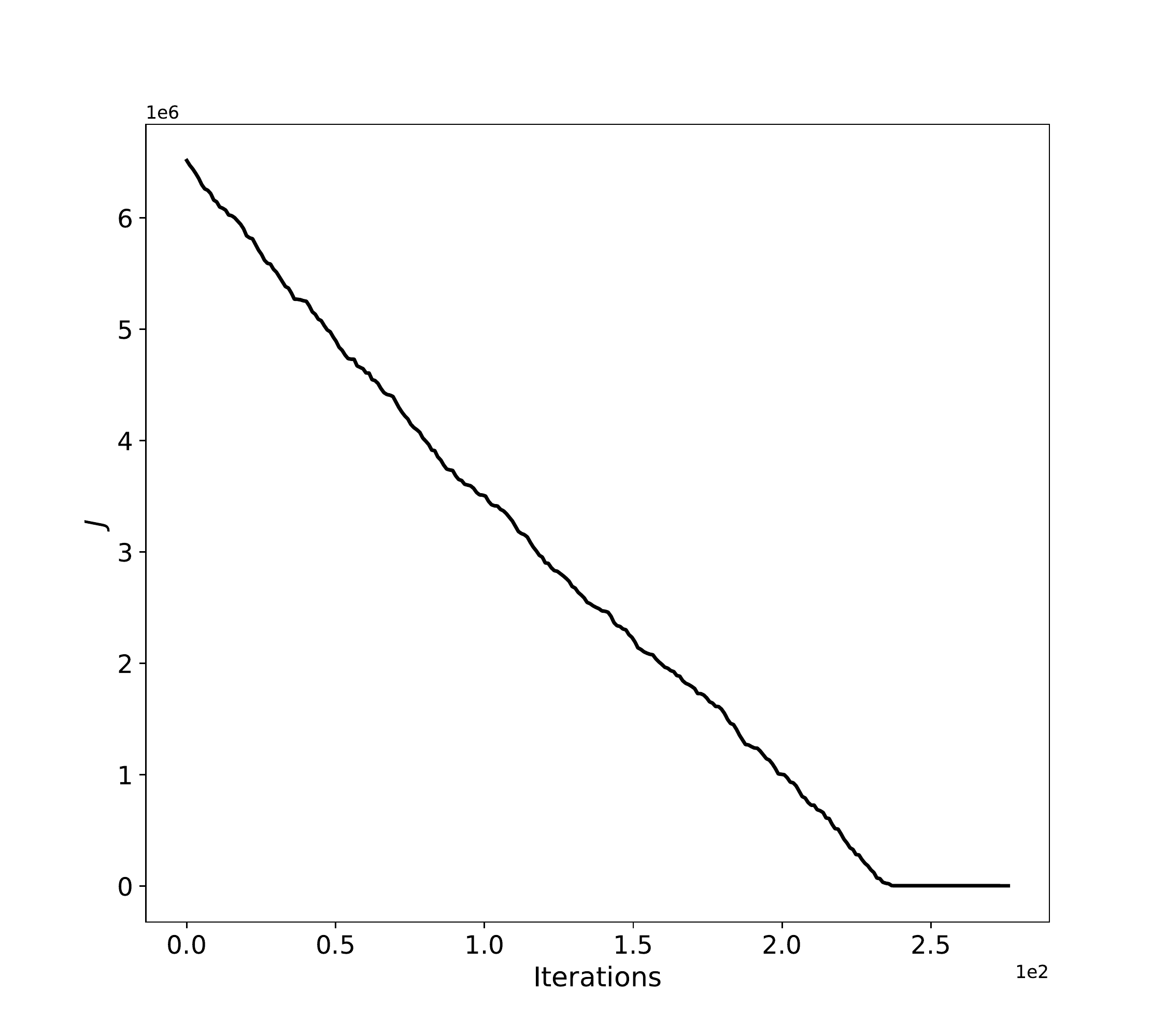}
	\caption{The above simulated annealing algorithm converging to 
	the global minimum of the action given in Eq.\ \ref{eq:power-grid-neighborhood}.
	In this case, $x \in M^{6594 \times 6594 }(\R^{\geq 1})$.
	}
	\label{fig:sim-ann-conv}
\end{figure}

\newpage

\bibliography{example}
\bibliographystyle{unsrt}

\end{document}